\documentclass[11pt,reqno]{amsart}

\allowdisplaybreaks

\usepackage{enumitem}

\usepackage{graphicx}
\usepackage{amscd,amsthm,amstext}
\usepackage{amsmath,amssymb,color}
\usepackage{amsfonts}
\usepackage{amssymb,appendix,wasysym}
\usepackage[latin1]{inputenc}
\usepackage[colorlinks=true]{hyperref}
\usepackage{mathrsfs}
\usepackage{yfonts}
\newtheorem{theorem}{Theorem}[section]

\newtheorem{lemma}[theorem]{Lemma}
\newtheorem{proposition}[theorem]{Proposition}

\newtheorem{remark}{Remark}[section]

\newcommand{\be}{\begin{equation}}
\newcommand{\ee}{\end{equation}}
\numberwithin{equation}{section}

\allowdisplaybreaks

\makeatletter
\@namedef{subjclassname@2020}{%
  \textup{2020} Mathematics Subject Classification}
\makeatother
\begin{document}

\title[Mixed local-nonlocal quasilinear problems with critical nonlinearities]
{Mixed local-nonlocal quasilinear problems with critical nonlinearities}

\author[J.V. da Silva]{Jo\~ao Vitor da Silva}
\address[J.V. da Silva]{Departamento de Matem\'atica, Universidade Estadual de Campinas, IMECC,
Rua S\'ergio Buarque de Holanda 651, Campinas, CEP 13083-859 Brazil}
\email{jdasilva@unicamp.br}

\author[A. Fiscella]{Alessio Fiscella}
\address[A. Fiscella]{Dipartimento di Matematica e Applicazioni, Universit\`a degli Studi di Milano-Bicocca, Via Cozzi 55, Milano, CAP 20125, Italy}
\email{alessio.fiscella@unimib.it}

\author[V.A.B. Viloria]{Victor A. Blanco Viloria}
\address[V.A.B. Viloria]{Departamento de Matem\'atica, Universidade Estadual de Campinas, IMECC,
Rua S\'ergio Buarque de Holanda 651, Campinas, CEP 13083-859 Brazil}
\email{vblancocc@gmail.com}


\begin{abstract}

We study existence and multiplicity of nontrivial solutions of the following problem
$$
\left\{
\begin{array}{rcll}
-\Delta_p u+(-\Delta_p)^{s} u & = & \lambda|u|^{q-2}u+|u|^{p^{\ast}-2}u & \mbox{ in }\Omega,\\
u & = & 0 & \mbox{ on }  \mathbb{R}^{N} \setminus \Omega,
\end{array}
\right.
$$
where $\Omega\subset \mathbb{R}^N$ is a bounded open set with smooth boundary, dimension $N\geq 2$, parameter $\lambda>0$, exponents $0<s<1<p<N$, while $q\in(1,p^{\ast})$ with $p^{\ast}=\frac{Np}{N-p}$. The problem is driven by an operator of mixed order obtained by the sum of the classical $p$-Laplacian and of the fractional $p$-Laplacian. We analyze three different scenarios depending on exponent $q$. For this, we combine variational methods with some topological techniques, such as the Krasnoselskii genus and the Lusternik-Schnirelman category theories.

\end{abstract}

\maketitle

\section{Introduction}\label{sec:introduction}

In the present paper we deal with the following problem
\begin{equation}\label{P}
\left\{
\begin{array}{rcll}
-\Delta_p u+(-\Delta_p)^{s} u & = & \lambda|u|^{q-2}u+|u|^{p^{\ast}-2}u & \mbox{ in }\Omega,\\
u & = & 0 & \mbox{ on }  \mathbb{R}^{N} \setminus \Omega,
\end{array}
\right.
\end{equation}
where $\Omega\subset\mathbb R^N$ is a bounded open set with smooth boundary, dimension $N\geq2$, parameter $\lambda>0$, exponents $0<s<1<p<N$, while $q\in(1,p^{\ast})$ with $p^{\ast}=\frac{Np}{N-p}$.
In the left hand side of \eqref{P} we have the classical $p$-Laplace operator $\Delta_p u=\text{div}(|\nabla u|^{p-2}\nabla u)$ plus its fractional counterpart which, up to normalization factor, can be set as
$$
(-\Delta_p)^su(x)= \lim_{\varepsilon\to0^+}\int_{\mathbb{R}^N\setminus B_\varepsilon(x)}\frac{|u(x)-u(y)|^{p-2}(u(x)-u(y))}{|x-y|^{N+sp}}dy,
$$
along any $u\in C_0^\infty(\Omega)$.

Recently, the interest on nonlinear problems driven by operators of mixed type has grown more and more, in connection with the study of optimal animal foraging strategies, see for example \cite{DV}. We should also highlight the relevant contributions in the following topics: Hong-Krahn-Szeg\"{o} inequality \cite{BDVV23}, Faber-Krahn inequality and applications \cite{BDVV23-1}, system of eigenvalue problems and their asymptotic limit as $p \to \infty$ \cite{BDD}, mixed models with concave-convex nonlinearities and their asymptotic limit as $p \to \infty$ \cite{DaSS20}, local/non-local equations of H\'enon type \cite{SalVec22}, regularity theory for mixed models \cite{DeFM22}, \cite{GK22} and \cite{SVWZ22}, and higher H\"{o}lder regularity for mixed local and nonlocal
elliptic models \cite{GL23}. Mathematically speaking, this operator offers quite relevant challenges caused by the combination of nonlocal difficulties with the lack of invariance under scaling.

In this context, \eqref{P} can be considered as a "Br\'ezis-Nirenberg" problem of mixed local-nonlocal type, being a generalization of the classical one studied in \cite{BN}. Actually, the problem in \cite{BN} was generalized in a quasilinear situation in \cite{GA2}  and in the fractional case in \cite{SV}. We also recall \cite{BDVV} where the authors studied \eqref{P} when $p=2$, proving the existence of a nontrivial solution of \eqref{P} when $q\in[p,p^{\ast})$. The main aim of the present paper is to generalize \cite{BDVV} with $p\not=2$ and $q\in(1,p^{\ast})$, providing existence and multiplicity results for \eqref{P} depending on the value of $q$.

We first study \eqref{P} in the sublinear case, namely with $q\in(1,p)$. By topological arguments based on the Krasnoselskii genus theory, we can state the following result.

\begin{theorem} \label{T3}
Let $0<s<1<p<N$ and let $q\in(1,p)$.

Then, there exists $\lambda_{\ast}>0$ such that for any $\lambda\in(0,\lambda_{\ast})$ problem \eqref{P} admits infinite many nontrivial solutions with negative energy.
\end{theorem}

Theorem \ref{T3} generalizes the multiplicity result proved in the classical case in \cite[Theorem 4.5]{GA2}.

In order to study \eqref{P} in the linear case, we apply the \textit{classical cohomological index theory} as in \cite{morse}. For this, we first need to introduce the eigenvalues of nonlinear operator $\mathcal{L}_{p, s}(\cdot):=-\Delta_{p}(\cdot)+(-\Delta_p)^s(\cdot)$. We recall here the eigenvalue problem
\begin{equation}\label{kir1}
\left\{
\begin{array}{rcll}
-\Delta_p u+(-\Delta_p)^{s} u & = & \lambda|u|^{p-2}u & \mbox{ in }\Omega,\\
u & = & 0 & \mbox{ on }  \mathbb{R}^{N} \setminus \Omega.
\end{array}
\right.
\end{equation}
Then, we set as $\{\lambda_k\}_k$ the sequence of eigenvalues nontrivially solving \eqref{kir1} with parameter $\lambda=\lambda_k$. We also recall the best constant of Sobolev embedding for $W_0^{1,p}(\Omega)$ as
\begin{equation}\label{sob1}
\mathcal{S}=\inf_{\substack{v\in W^{1,p}_0(\Omega)\\
v\not\equiv0}}\displaystyle\frac{\|\nabla v\|_p^p}{\|v\|_{p^{\ast}}^{p}},
\end{equation}
which is fundamental to state the range of $\lambda$ in the following bifurcation result.

\begin{theorem}\label{T2}
Let $0<s<1<p<N$ and let $q=p$. Let $\widetilde{\lambda}=\frac{\mathcal{S}}{|\Omega|^{p/N}}$, where $\mathcal{S}$ is set in \eqref{sob1}.

Then, we have:
\begin{itemize}
\item [$(i)$] if $\lambda_{k+1}-\widetilde{\lambda}<\lambda<\lambda_1$, problem \eqref{P} admits a pair of nontrivial solutions $\pm u_{\lambda}$ such that $u_{\lambda}\to 0$ when $\lambda\to \lambda_1$;
\item [$(ii)$] if $\lambda_k\leq \lambda<\lambda_{k+1}=\ldots=\lambda_{k+m}<\lambda_{k+m+1}$ for some  $k$, $m\in \mathbb{N}$ and $\lambda>\lambda_{k+1}-\widetilde{\lambda}$,  problem \eqref{P} admits $m$ distinct pairs of nontrivial solutions $\pm u_{j,\lambda}$, with $j=1,\ldots,m$, such that $u_{j,\lambda}\to 0$ when $\lambda\to \lambda_{k+1}$.
\end{itemize}
\end{theorem}

Theorem \ref{T2} complements the multiplicity result proved in \cite[Theorem 1.1]{PSY}.

For the superlinear case, with $q\in(p,p^{\ast})$, we need to set the parameters
\begin{equation}\label{parametri}
m_{N,p,s}:=\min\left\{\frac{N-p}{p-1},p(1-s)\right\},\qquad\beta_{N,q,p}:=N-q\frac{N-p}{p}.
\end{equation}
Thus, we have the following multiplicity result, strongly depending on $\Omega$.

\begin{theorem}\label{T1}
Let $N>p$, $0<s<1$ and $2\leq p<q<p^{\ast}$ satisfy $m_{N,p,s}>\beta_{N,q,p}$.

Then, there exists $\lambda_{{\ast}{\ast}}>0$ such that for any $\lambda\in(0,\lambda_{{\ast}{\ast}})$ problem \eqref{P} admits $\text{cat}_\Omega(\Omega)$ nontrivial solutions.
\end{theorem}

For the proof of Theorem \ref{T1}, we exploit the Lusternik-Schnirelman category theory. Theorem \ref{T1} complements \cite[Theorem 1]{AD} and it is consistent with the pure fractional case in \cite[Theorem 1.1]{FMS} with $p=2$. However, the presence of local terms in \eqref{P} affects the multiplicity in Theorem \ref{T1}. Indeed, when $p=2$, if we consider $N>4-2s$ (which occurs when either $N\geq4$ or $N=3$ and $s>1/2$) then $m_{N,2,s}:=\min\left\{N-2,2(1-s)\right\}=2(1-s)$, so that
$$
m_{N,2,s}>\beta_{N,q,2}\quad\Longrightarrow\quad q>2^{\ast}-\frac{2(2-2s)}{N-2}.
$$
On the other hand, in \cite[Theorem 1.1]{FMS} the authors assume that
$$
N>\frac{2s(q+2)}{q}\quad\Longrightarrow\quad q>2^{\ast}_s-2,
$$
with $2^{\ast}_s=\frac{2N}{N-2s}$. By a simple calculation, we can see that
$$
2^{\ast}-\frac{2(2-2s)}{N-2}>2^{\ast}_s-2,
$$
whenever $N>4-2s$.

Last but not least, we want to investigate \eqref{P} when $\lambda$ is sufficiently large. Indeed, in Theorem \ref{T1} we need to control $\lambda$ in a suitable bounded range. However, we can state the following existence result, covering also the case when $1<p<2$.

\begin{theorem}\label{supex}
Let $N>p$ and $0<s<1<p<q<p^{\ast}$.

Then, problem \eqref{P} admits at least one nontrivial solution, under the following cases:
\begin{itemize}
\item [$(i)$] if $m_{N,p,s}>\beta_{N,q,p}$, for any $\lambda>0$;
\item [$(ii)$] if $m_{N,p,s}\leq\beta_{N,q,p}$, for any $\lambda\geq\lambda^{\ast}$ with a suitable $\lambda^{\ast}>0$.
\end{itemize}
\end{theorem}

We observe that the  dichotomy in Theorem \ref{supex} is strongly affected by the presence of the nonlocal term in \eqref{P}. Indeed, when $s=1$ we get $m_{N,p,1}=0$ so that always $m_{N,p,1}\leq\beta_{N,q,p}$ from \eqref{parametri}, since $q<p^{\ast}$.
Furthermore, as well explained in Remark \ref{osservazione}, the joint requests $p<q<p^{\ast}$ and $m_{N,p,s}>\beta_{N,q,p}$ imply that
$$
q>\max\left\{p,p^{\ast}-\frac{p}{p-1}\right\},
$$
which is assumed in \cite[Theorem 3.3]{GA2}. For this, Theorem \ref{supex} complements the classical situation contained in \cite[Theorems 3.2 and 3.3]{GA2} and generalizes \cite[Theorem 1.5]{BDVV}.

The paper is organized as follows. In Section \ref{sec2} we introduce the variational setting of \eqref{P} and we provide the compactness property of the related energy functional. In Section \ref{sec3} we prove Theorem \ref{T3}. In Section \ref{sec4} we first set the eigenvalues of \eqref{kir1} and then we prove Theorem \ref{T2}. In Section \ref{sec5} we study the existence result of Theorem \ref{supex}. In Section \ref{sec6} we deal with Theorem \ref{T1}.

\section{Variational setting}\label{sec2}

From now on, without further mentioning, we suppose that $\Omega\subset\mathbb R^N$ is an open bounded set with smooth boundary and dimension $N\geq2$. For any $r\in[1,\infty)$ and any $E\subseteq\mathbb R^N$, we denote by $L^r(E)$ and $W^{1,r}(E)$ the usual Lebesgue and Sobolev spaces, endowed with the norms $\|\cdot\|_r$ and $\|\cdot\|_r+\|\nabla\cdot\|_r$, respectively.

We recall the homogeneous Sobolev space $W_0^{1,p}(\Omega)$ as the closure of $C_0^\infty(\Omega)$ with respect to the norm $\|\nabla\cdot\|_p$. However, in order to handle the boundary condition in \eqref{P}, we need to work with functions set in the whole space $\mathbb R^N$. For this, $W_0^{1,p}(\Omega)$ is not enough and we consider the following functional space
$$
\mathrm{X}_0(\Omega)=\left\{u\in W^{1,p}(\mathbb R^N):\,\,u|_{\Omega}\in W_0^{1,p}(\Omega)\mbox{ and }u=0\mbox{ a.e. in }\mathbb R^N\setminus\Omega\right\}
$$
endowed with the norm
$$
\|u\|_{\mathrm{X}_0}=\left(\|\nabla u\|_p^p+[u]_{s,p}^p\right)^{1/p},
$$
where $\left[\,\cdot\,\right]_{s,p}$ stands for the Gagliardo seminorm, given by
$$
\displaystyle [u]_{s,p}=\left(\iint_{\mathbb R^{2N}}\frac{|u(x)-u(y)|^p}{|x-y|^{N+ps}}\,dx\,dy\right)^{1/p}.
$$
It is well known that $\mathrm{X}_0(\Omega)$ is a real separable and uniformly convex space. Also, by \cite[Lemma 2.1]{BDD} we have the following relation between the gradient seminorm and the Gagliardo one.

\begin{lemma}\label{lemmaBDD}
There exists a constant $\mathrm{c}=\mathrm{c}(N,p,s,\Omega)>0$ such that
$$
\displaystyle \iint_{\mathbb R^{2N}}\frac{|u(x)-u(y)|^p}{|x-y|^{N+ps}}\,dx\,dy\leq \mathrm{c}\int_\Omega|\nabla u|^p\,dx,\qquad\mbox{for any }u\in \mathrm{X}_0(\Omega).
$$
\end{lemma}

A function $u\in \mathrm{X}_0(\Omega)$ is called a (weak) solution of problem \eqref{P} if
$$
\displaystyle \int_\Omega|\nabla u|^{p-2}\nabla u\cdot\nabla\varphi \,dx+\iint_{\mathbb R^{2N}}\mathcal{A}u(x,y)\left[\varphi(x)-\varphi(y)\right]\,dx\,dy
=\lambda\int_\Omega |u|^{q-2}u\varphi \,dx+\int_\Omega |u|^{p^{\ast}-2}u\varphi \,dx
$$
for any $\varphi\in \mathrm{X}_0(\Omega)$, where for simplicity
\begin{equation}\label{semigagliar}
 \mathcal{A}u(x,y):=\frac{|u(x)-u(y)|^{p-2}(u(x)-u(y))}{|x-y|^{N+ps}}.
\end{equation}
Hence, problem \eqref{P} has a variational structure, since the (weak) solutions of \eqref{P} are critical points of $J_\lambda:\mathrm{X}_0(\Omega)\to\mathbb R$ set as
$$
J_\lambda(u):=\frac{1}{p}\|\nabla u\|_{\mathrm{X}_0}^p-\frac{\lambda}{q}\|u\|_q^q-\frac{1}{p^{\ast}}\|u\|_{p^{\ast}}^{p^{\ast}}.
$$
We conclude this section studying the compactness properties of $J_\lambda$.
We say that $J_\lambda$ satisfies the Palais-Smale condition at level $c\in\mathbb R$, \textnormal{(PS)$_c$} for short, if any sequence $\{u_n\}_n\subset \mathrm{X}_0(\Omega)$ such that
\begin{align}\label{palais-smale}
	J_\lambda(u_n)\to c \quad \text{and}\quad J^{\prime}_\lambda(u_n)\to 0 \quad\text{in }\left( \mathrm{X}_0(\Omega)\right)^{\ast}\quad\text{as }n\to\infty,
\end{align}
admits a convergent subsequence in $\mathrm{X}_0(\Omega)$.

In order to prove the $(PS)_c$ condition at a suitable range for $c$, we first need to handle the gradient seminorm appearing in $\|\cdot\|_{\mathrm{X}_0}$.

\begin{lemma}\label{lem2.3.01}
Let $\lambda>0$ and let $\{u_n\}_n\subset \mathrm{X}_0(\Omega)$ be a bounded sequence satisfying \eqref{palais-smale} with $c\in\mathbb R$. Then, up to a subsequence, $\nabla u_n(x)\to\nabla u(x)$ a.e. in $\Omega$ as $n\to\infty$.
\end{lemma}
\begin{proof}
Since $\{u_n\}_n$ is bounded in $\mathrm{X}_0(\Omega)$, there exists a subsequence, still denoted by $\{u_n\}_n$, and $u\in \mathrm{X}_0(\Omega)$ such that
\begin{equation}\label{ter1.1}
\begin{array}{ll}
u_n\rightharpoonup u \quad \text{ in } \quad \mathrm{X}_0(\Omega),\quad & \nabla u_n\rightharpoonup\nabla u \mbox{ in } \left[L^p(\Omega)\right]^N, \\
|u_n(x)|\leq h(x) \,\text{ a.e. in }\, \Omega,\\
u_n\rightarrow u\mbox{ in } L^r(\Omega), & u_n(x)\to u(x)\hspace{0.5mm}\text{ a.e. in } \Omega,
\end{array}
\end{equation}
as $n\to\infty$ with $r\in[1,p^{\ast})$ and $h\in L^p(\Omega).$

For any $\kappa\in\mathbb N$, let $T_\kappa:\mathbb{R}\to\mathbb{R}$ be the truncation function defined by
$$
		T_\kappa(t):=
		\begin{cases}
			t & \text{if } |t|\leq \kappa,\\[1ex]
			\displaystyle \kappa\frac{t}{|t|} &\text{if }|t|> \kappa.
		\end{cases}
$$
Let $\kappa\in\mathbb N$  be fixed.
Then, by \eqref{palais-smale} we have
\begin{align}\label{ineq1}
\text{o}_n (1)= & \left\langle J_\lambda^{\prime}\left(u_n\right), T_\kappa\left(u_n-u\right)\right\rangle \nonumber \\
= & \int_{\Omega}\left|\nabla u_n\right|^{p-2} \nabla u_n \cdot \nabla T_\kappa\left(u_n-u\right) dx
+\iint_{\mathbb{R}^{2N}}\mathcal{A}u_n(x,y)[T_\kappa(u_n-u)(x)-T_\kappa(u_n-u)(y)]\,dx\,dy\nonumber\\
& -\lambda \int_{\Omega}\left|u_n\right|^{q-2} u_n T_\kappa\left(u_n-u\right) dx-\int_{\Omega}\left|u_n\right|^{p^{\ast}-2} u_n T_\kappa\left(u_n-u\right) dx ,
\end{align}
as $n\to\infty$, since $\{T_k(u_n-u)\}_n$ is bounded in $\mathrm{X}_0(\Omega)$.
By H\"older's inequality and \eqref{ter1.1} we get
\begin{equation}
\begin{aligned}\label{nul1}
   &\lim_{n \rightarrow \infty} \int_{\Omega}|\nabla u|^{p-2} \nabla u \cdot \nabla T_\kappa\left(u_n-u\right)dx=0
   \\
   &\lim_{n \rightarrow \infty} \iint_{\mathbb{R}^{2N}}\mathcal{A}u(x,y)[T_\kappa(u_n-u)(x)-T_\kappa(u_n-u)(y)]\,dx\,dy=0.
\end{aligned}
\end{equation}
By the boundedness of $\{u_n\}_n$ and \eqref{sob1} we have
\begin{equation}\label{ineq2}
\left|\int_{\Omega} u_n^{p^{\ast}-2} u_n T_\kappa\left(u_n-u\right) dx \right|\leq \kappa \int_{\Omega} |u_n|^{p^{\ast}-1} dx \leq C \kappa,
\end{equation}
with a constant $C>0$ independent of $n$ and $\kappa$.
Thus, by using  \eqref{ineq1}-\eqref{ineq2} we get
\begin{align}\label{equality1}
\limsup _{n \rightarrow \infty} & {\left[\int_{\Omega}\left[\left|\nabla u_n\right|^{p-2} \nabla u_n-|\nabla u|^{p-2} \nabla u\right] \cdot \nabla T_\kappa\left(u_n-u\right) dx\right.} \nonumber\\
& \left.+\iint_{\mathbb{R}^{2N}}\left[\mathcal{A}u_n(x,y)-\mathcal{A}u(x,y)\right]\left[T_\kappa(u_n-u)(x)-T_\kappa(u_n-u)(y)\right]\,dx\,dy\right] \\
= & \limsup _{n \rightarrow \infty} \int_{\Omega}\left|u_n\right|^{p^{\ast}-2} u_n T_\kappa\left(u_n-u\right) dx \nonumber.
\end{align}
Now, we prove the following claim.

\vspace{0.3cm}
\noindent\textbf{Claim.} {\em For any $(x,y)\in\mathbb R^{2N}$ we have that
\begin{equation}\label{ineq3}
\left[\mathcal{A}u_n(x,y)-\mathcal{A}u(x,y)\right]\left[T_\kappa(u_n-u)(x)-T_\kappa(u_n-u)(y)\right]\geq 0.
\end{equation}}

\noindent
For this, we split $\mathbb R^{2N}$ in 4 sets
$$
\begin{aligned}
S_1&=\left\{(x,y)\in\mathbb R^{2N}:\,\,|(u_n-u)(x)|\leq\kappa,\quad|(u_n-u)(y)|\leq\kappa\right\}\\
S_2&=\left\{(x,y)\in\mathbb R^{2N}:\,\,|(u_n-u)(x)|\leq\kappa<|(u_n-u)(y)|\right\}\\
S_3&=\left\{(x,y)\in\mathbb R^{2N}:\,\,|(u_n-u)(y)|\leq\kappa<|(u_n-u)(x)|\right\}\\
S_1&=\left\{(x,y)\in\mathbb R^{2N}:\,\,|(u_n-u)(x)|>\kappa,\quad|(u_n-u)(y)|>\kappa\right\}
\end{aligned}
$$
and we study case by case.

\vspace{0.3cm}
$\bullet$\, {\em Case 1. Let $(x,y)\in S_1$.} We have that
$$
\begin{aligned}
\left[\mathcal{A}u_n(x,y)-\mathcal{A}u(x,y)\right]&\left[T_\kappa(u_n-u)(x)-T_\kappa(u_n-u)(y)\right]\\
&=\left[\mathcal{A}u_n(x,y)-\mathcal{A}u(x,y)\right]\left[(u_n-u)(x)-(u_n-u)(y)\right]
\geq0
\end{aligned}
$$
where the last inequality follows by the well-known Simon's inequalities, see \cite[formula (2.2)]{Simon-1978}.

\vspace{0.3cm}
$\bullet$\, {\em Case 2. Let $(x,y)\in S_2$.} From $|(u_n-u)(x)|\leq\kappa<|(u_n-u)(y)|$, we have
$$
T_\kappa(u_n-u)(x)-T_\kappa(u_n-u)(y)=\left\{
\begin{matrix}
(u_n-u)(x)-\kappa\leq0,& {\mbox{ if $u_n(x)\geq u(x)$ and $u_n(y)\geq u(y)$}}\\
(u_n-u)(x)+\kappa\geq0,& {\mbox{ if $u_n(x)\geq u(x)$ and $u_n(y)<u(y)$}}\\
(u_n-u)(x)-\kappa\leq0,& {\mbox{ if $u_n(x)< u(x)$ and $u_n(y)\geq u(y)$}}\\
(u_n-u)(x)+\kappa\geq0,& {\mbox{ if $u_n(x)< u(x)$ and $u_n(y)< u(y)$}}.
\end{matrix}
\right.
$$
While, since $-|(u_n-u)(x)|\leq|(u_n-u)(x)|<|(u_n-u)(y)|$ and considering the monotonicity of function $h(t)=|t|^{p-2}t$, we get
$$
\left\{
\begin{matrix}
\mathcal{A}u_n(x,y)-\mathcal{A}u(x,y)\leq0,& {\mbox{ if $u_n(x)\geq u(x)$ and $u_n(y)\geq u(y)$}}\\
\mathcal{A}u_n(x,y)-\mathcal{A}u(x,y)\geq0,& {\mbox{ if $u_n(x)\geq u(x)$ and $u_n(y)<u(y)$}}\\
\mathcal{A}u_n(x,y)-\mathcal{A}u(x,y)\leq0,& {\mbox{ if $u_n(x)< u(x)$ and $u_n(y)\geq u(y)$}}\\
\mathcal{A}u_n(x,y)-\mathcal{A}u(x,y)\geq0,& {\mbox{ if $u_n(x)< u(x)$ and $u_n(y)< u(y)$}}.
\end{matrix}
\right.
$$
Thus, summing up we obtain \eqref{ineq3}.

\vspace{0.3cm}
$\bullet$\, {\em Case 3. Let $(x,y)\in S_3$.} It is specular to the Case 2.

\vspace{0.3cm}
$\bullet$\, {\em Case 4. Let $(x,y)\in S_4$.} It is trivial, since
$$
T_\kappa(u_n-u)(x)-T_\kappa(u_n-u)(y)=\kappa-\kappa=0.
$$
This concludes the proof of the Claim.

\vspace{0.3cm}
Then, by \eqref{ineq3} we obtain
\begin{align}\label{ineq4}
 \nonumber   &\limsup _{n \rightarrow \infty} \int_{\Omega}\left[\left|\nabla u_n\right|^{p-2} \nabla u_n-|\nabla u|^{p-2} \nabla u\right] \cdot \nabla T_\kappa\left(u_n-u\right) dx \\
&\leq \limsup _{n \rightarrow \infty} \int_{\Omega}\left|u_n\right|^{p^{\ast}-2} u_n T_\kappa\left(u_n-u\right) dx \\ \nonumber
&\leq C \kappa.
\end{align}
We set
$$e_n(x):=\left[|\nabla u_n(x)|^{p-2}\nabla u_n(x)-|\nabla u(x)|^{p-2}\nabla u(x)\right]\cdot\nabla(u_n(x)-u(x)).$$
By using again the Simon's inequalities in \cite[formula (2.2)]{Simon-1978}, we see that $e_n(x)\geq 0$ a.e.\, in $\Omega$.
We split $\Omega$ by
$$
		S_n^\kappa=\left\{x\in\Omega:\,\,|u_n(x)-u(x)|\leq \kappa\right\}
		\quad\text{and}\quad
		G_n^\kappa=\left\{x\in\Omega:\,\,|u_n(x)-u(x)|> \kappa\right\},
$$
where $n$, $\kappa \in \mathbb N$ are fixed.
Taking $\theta\in(0,1)$  and using H\"older's inequality as well as the boundedness of $\{e_n\}_n$ along with \eqref{ineq4} gives
\begin{align*}
    \int_{\Omega}e_n^{\theta}\,dx &\leq \left( \int_{S_n^{\kappa}}e_n\, dx\right)^{\theta}|S_n^{\kappa}|^{1-\theta}+ \left( \int_{G_n^{\kappa}}e_n\, dx\right)^{\theta}|G_n^{\kappa}|^{1-\theta}\\
    &\leq (\kappa C)^{\theta} |S_n^{\kappa}|^{1-\theta}+\widetilde{C}^{\,\theta} |G_n^{\kappa}|^{1-\theta}.
\end{align*}
From this, noticing that  $|G_n^{\kappa}|\to 0$ as $n\to \infty$, we get
\begin{equation*}
    0\leq \limsup_{n\to \infty} \int_{\Omega}e_n^{\theta} \,dx\leq (\kappa C)|\Omega|^{1-\theta}.
\end{equation*}
Letting $\kappa\to 0^{+}$, we obtain that $e_n^\theta\to0$ in $L^1(\Omega)$ as $n\to\infty$. Thus, we may assume that $e_n(x)\to 0$ a.e. \,in $\Omega$. Applying Simon's inequalities in \cite[formula (2.2)]{Simon-1978}, we prove the assertion of the lemma.
\end{proof}

In order to apply \eqref{lem2.3.01}, we need the boundedness of $(PS)_c$  sequences, as proved in what follows.

\begin{lemma}\label{limitata}
Let $\lambda>0$. Then, any sequence $\{u_n\}_n\subset \mathrm{X}_0(\Omega)$ satisfying \eqref{palais-smale}, with $c\in\mathbb R$, is bounded.
\end{lemma}
\begin{proof}
Let $\{u_n\}_n\subset \mathrm{X}_0(\Omega)$ satisfy \eqref{palais-smale}. We divide the proof depending on the value of $q$.

\vspace{0.3cm}
$\bullet$\, {\em Case 1. Let $q\in(1,p)$.} By the Sobolev embedding
$$
J_\lambda\left(u_n\right)-\frac{1}{p^{\ast}}\left\langle J_\lambda^{\prime}\left(u_n\right), u_n\right\rangle
\geq \left(\frac{1}{p}-\frac{1}{p^{\ast}}\right)\left\|u_n\right\|_{\mathrm{X}_0}^p-\lambda C_q\|u_n\|_{\mathrm{X}_0}^q.
$$
From this and \eqref{palais-smale}, there exists $d_\lambda>0$ such that
$$c+d_\lambda\|u_n\|_{\mathrm{X}_0}+\text{o}_n(1)\geq  \left(\frac{1}{p}-\frac{1}{p^{\ast}}\right)\|u_n\|_{\mathrm{X}_0}^p-\lambda C_q\|u_n\|_{\mathrm{X}_0}^q,$$
which yields the boundedness, since $p^{\ast}>p>q>1.$

\vspace{0.3cm}
$\bullet$\, {\em Case 2. Let $q=p$.} Let us assume by contradiction that $\{u_n\}_n$ is unbounded. Then, up to a subsequence, $\|u_n\|_{\mathrm{X}_0}\to \infty$ as $n\to \infty$, and we can assume without loss of generality that $u_n\not\equiv 0$. We observe that  $\big\{u_n\|u_n\|_{\mathrm{X}_0}^{-1}\big\}_n$ is bounded in $\mathrm{X}_0(\Omega)$, so by \eqref{palais-smale}
\begin{align*}
\text{o}_n(1)= \left\langle J^{\prime}_\lambda(u_n),u_n\|u_n\|_{\mathrm{X}_0}^{-1}\right\rangle &=\displaystyle\frac{\|u_n\|_{\mathrm{X}_0}^{ p}-\lambda\|u_n\|_{ p}^{ p}-\|u_n\|_{p^{\ast}}^{p^{\ast}}}{\|u_n\|_{\mathrm{X}_0}}
  \\[0.3cm]
  &=\displaystyle\frac{ p J_\lambda(u_n)+\displaystyle\left(\frac{ p}{p^{\ast}}-1\right)\|u_n\|_{p^{\ast}}^{p^{\ast}}}{\|u_n\|_{\mathrm{X}_0}},\quad\mbox{as}\quad n\to \infty.
\end{align*}
From this combined with \eqref{palais-smale} we get
\begin{align}\label{limites}
    \displaystyle\frac{\|u_n\|_{p^{\ast}}^{p^{\ast}}}{\|u_n\|_{\mathrm{X}_0}}&\to 0,\quad\mbox{as}\quad n\to \infty,\quad
\end{align}
and since $\|u_n\|_{\mathrm{X}_0}\to \infty$ as $n\to \infty$, we have
 \begin{equation}\label{limitess}
       \displaystyle\frac{\|u_n\|_{p^{\ast}}^{p^{\ast}}}{\|u_n\|_{\mathrm{X}_0}^{ p}}\to 0,\quad\mbox{as}\quad n\to \infty.
 \end{equation}
Similarly,
 \begin{equation}\label{limitess2}
       \displaystyle\frac{\|u_n\|_{p^{\ast}}^{ p}}{\|u_n\|_{\mathrm{X}_0}^{ p}}
			=\displaystyle\frac{\|u_n\|_{p^{\ast}}^{ p}}{\|u_n\|_{\mathrm{X}_0}^{ p/p^{\ast}}}\frac{1}{\|u_n\|_{\mathrm{X}_0}^{ p(p^{\ast}-1)/p^{\ast}}}\to 0,\quad\mbox{as}\quad n\to \infty.
 \end{equation}
Finally, we show the following estimate
 \begin{equation}\label{limitess3}
       \displaystyle\frac{\|u_n\|_{ p}^{ p}}{\|u_n\|_{\mathrm{X}_0}^{ p}}\to 0,\quad\mbox{as}\quad n\to \infty.
 \end{equation}
By H\"older's inequality we have
\begin{align*}
    \|u_n\|_{ p}^{ p}&=\int_{\Omega} |u_n|^{ p}\,dx
    \leq \left(\int_{\Omega} |u_n|^{p^{\ast}} \,dx\right)^{ p/p^{\ast}} |\Omega|^{1- p/p^{\ast}}
    \\[0.3cm]
    &=\|u_n\|_{p^{\ast}}^{ p} \,|\Omega|^{1- p/p^{\ast}},
\end{align*}
so that
\begin{equation}\label{para0}
\displaystyle\frac{\|u_n\|_{ p}^{ p}}{\|u_n\|_{\mathrm{X}_0}^{ p}}\leq \displaystyle\frac{\|u_n\|_{p^{\ast}}^{ p}}{{\|u_n\|_{\mathrm{X}_0}^{ p}}}\,|\Omega|^{1-p/p^{\ast}}\to 0,\quad \mbox{as}\quad n\to\infty.
\end{equation}
By \eqref{palais-smale} and the fact that $\|u_n\|_{\mathrm{X}_0}\to \infty$, we get
$$
\left\langle J^{\prime}_\lambda(u_n),u_n\|u_n\|_{\mathrm{X}_0}^{- p} \right\rangle \to 0,\quad\mbox{as}\quad n\to \infty.
$$
Thus, by using \eqref{limitess} and \eqref{limitess3}
\begin{align*}
0=\lim_{n\to \infty} \left\langle  J^{\prime}_\lambda(u_n),u_n\|u_n\|_{\mathrm{X}_0}^{- p}\right\rangle &=
\lim_{n\to \infty}\displaystyle\frac{\|u_n\|_{\mathrm{X}_0}^{ p}-\lambda\|u_n\|_{ p}^{ p}-\|u_n\|_{p^{\ast}}^{p^{\ast}}}{\|u_n\|_{\mathrm{X}_0}^{ p}}
  \\[0.3cm]
  &=\lim_{n\to \infty} \left(1-\displaystyle \lambda \frac{\|u_n\|_{ p}^{p}}{\|u_n\|_{\mathrm{X}_0}^{ p}}-\displaystyle\frac{\|u_n\|_{ p^{\ast}}^{ p^{\ast}}}{\|u_n\|_{\mathrm{X}_0}^{ p}} \right)\\
  & =1,
\end{align*}
which yields the desired contradiction.

\vspace{0.3cm}
$\bullet$\, {\em Case 3. Let $q\in(p,p^{\ast})$.} By \eqref{palais-smale}, there exists $d_\lambda>0$ such that
$$
c+d_\lambda\|u_n\|_{\mathrm{X}_0}+\text{o}_n(1)\geq
J_\lambda\left(u_n\right)-\frac{1}{q}\left\langle J_\lambda^{\prime}(u_n), u_n\right\rangle
    \geq \left(\frac{1}{p}-\frac{1}{q}\right)\left\|u_n\right\|_{\mathrm{X}_0}^p
$$
which gives the desired result.
\end{proof}

\begin{lemma}\label{lemmaPS}
Let $\lambda>0$ and let $\mathcal{S}$ be as in \eqref{sob1}. Then, the $(PS)_c$ condition holds true under the following cases:
\begin{itemize}
\item [$(i)$] for any $c<\frac{1}{N}\mathcal{S}^{N/p}$, if $q\in[p,p^{\ast})$;
\item [$(ii)$] for any
$$c<\frac{1}{N}\mathcal{S}^{\frac{N}{p}}-|\Omega|\left(\frac{1}{p}-\frac{1}{p^{\ast}}\right)^{-\frac{q}{p^{\ast}-q}}\left[\lambda\left(\frac{1}{q}-\frac{1}{p}\right)\right]^{\frac{p^{\ast}}{p^{\ast}-q}}$$
if $q\in(1,p)$.
\end{itemize}
\end{lemma}
\begin{proof}
Since by Lemma \ref{limitata} the sequence $\{u_n\}_n$ is bounded in $\mathrm{X}_0(\Omega)$, there exists $u\in \mathrm{X}_0(\Omega)$ such that, up to a
subsequence still relabeled $\{u_n\}_n$, it follows that
\begin{equation}\label{ter1.2}
\begin{array}{ll}
u_n\rightharpoonup u \quad \text{ in }\quad \mathrm{X}_0(\Omega),\quad & \nabla u_n\rightharpoonup\nabla u \mbox{ in } (L^{p}(\Omega))^N, \\
\nabla u_n(x)\to \nabla u(x),\, &u_n(x)\to u(x)\hspace{0.5mm}\text{ a.e. in } \Omega,\\
u_n\rightarrow u\mbox{ in } L^r(\Omega), & \|u_n-u\|_{p^{\ast}}\to \ell \mbox{ for some }\ell\geq 0,
\end{array}
\end{equation}
as $n\to \infty$, with $r\in [1,p^{\ast})$. Hence, by \eqref{ter1.2}, Lemma \ref{lem2.3.01} and the Br\'ezis-Lieb Lemma, we get
\begin{equation}\label{ter1.3}
\begin{array}{ll}
&\|\nabla u_n\|_p^p-\|\nabla u_n-\nabla u\|_p^{p}=\|\nabla u\|_p^p+\text{o}_n(1)
 \\
&[u_n]_{s,p}^{p}-[u_n-u]_{s,p}^{p}=[u]_{s,p}^{p}+\text{o}_n(1)
\\
&\|u_n\|_{p^{\ast}}^{p^{\ast}}-\|u_n-u\|_{p^{\ast}}^{p^{\ast}}=\|u\|_{p^{\ast}}^{p^{\ast}}+\text{o}_n(1),
\end{array}
\end{equation}
as $n\to \infty$.

Furthermore, as shown in the proof of~\cite[Lemma~2.4]{CaPu}, by \eqref{ter1.2}
the sequence $\{\mathcal U_n\}_n$,
defined in $\mathbb R^{2N}\setminus\mbox{Diag\,}(\mathbb R^{2N})$ by
\begin{equation*}
(x,y)\mapsto\mathcal U_n(x,y)=\frac{|u_n(x)-u_n(y)|^{p-2}[u_n(x)-u_n(y)]}{|x-y|^{(N+ps)/p^{\prime}}},
\end{equation*}
is bounded in $L^{p^{\prime}}(\mathbb R^{2N})$ as well as $\mathcal U_n\to\mathcal U$ a.e. in $\mathbb R^{2N}$, where
\begin{equation*}
\mathcal U(x,y)
=\frac{|u(x)-u(y)|^{p-2}
[u(x)-u(y)]}{|x-y|^{(N+ps)/p^{\prime}}}.
\end{equation*}
Thus, going if necessary to a further subsequence, we get that $\mathcal U_n\rightharpoonup
\mathcal U$ in $L^{p^{\prime}}(\mathbb R^{2N})$ as $n\to\infty$ and so, using notation in \eqref{semigagliar}
\begin{equation}\label{x1}
\lim_{n\to\infty}\iint_{\mathbb R^{2N}}\mathcal Au_n(x,y)[\varphi(x)-\varphi(y)]\,dx\,dy=\iint_{\mathbb R^{2N}}\mathcal Au(x,y)[\varphi(x)-\varphi(y)]\,dx\,dy
\end{equation}
for any $\varphi\in \mathrm{X}_0(\Omega)$, since
$$|\varphi(x)-\varphi(y)|\cdot|x-y|^{-(N+ps)/p}\in L^p(\mathbb R^{2N}).$$

Thus, by \eqref{palais-smale}, \eqref{ter1.2} and \eqref{x1} we get
\begin{align*}
\text{o}_n (1)= & \left\langle J_\lambda^{\prime}\left(u_n\right), \left(u_n-u\right)\right\rangle  \\
= & \int_{\Omega}\left|\nabla u_n\right|^{p-2} \nabla u_n \cdot \nabla \left(u_n-u\right) dx
+\iint_{\mathbb{R}^{2N}}\mathcal Au_n(x,y)[(u_n-u)(x)-(u_n-u)(y)]\,dx\,dy\\
& -\lambda \int_{\Omega}\left|u_n\right|^{q-2} u_n \left(u_n-u\right) dx-\int_{\Omega}\left|u_n\right|^{p^{\ast}-2} u_n \left(u_n-u\right) dx
\\
=& \|\nabla u_n\|_p^p-\|\nabla u\|_p^p+[u_n]_{s,p}^p-[u]_{s,p}^p-\|u_n\|_{p^{\ast}}^{p^{\ast}}+\|u\|_{p^{\ast}}^{p^{\ast}}+\text{o}_n(1)
\end{align*}
as $n\to \infty$. Then, by using \eqref{ter1.2} and \eqref{ter1.3}, we have
\begin{equation}\label{ter1.4}
\lim_{n\to\infty}\left(\|\nabla u_n-\nabla u\|_p^p+[u_n-u]_{s,p}^p\right)=\lim_{n\to\infty}\|u_n-u\|_{p^{\ast}}^{p^{\ast}}=\ell^{p^{\ast}}.
\end{equation}
From this, by using \eqref{sob1} we get $\ell^{p^{\ast}}\geq \mathcal{S}\ell^{p}$.

Let us assume by contradiction that $\ell>0$, so that
\begin{equation}\label{ter1.5}
 \ell \geq \mathcal{S}^{\frac{1}{p^{\ast}-p}}.
\end{equation}
Now, by estimating
\begin{align*}
   J_\lambda(u_n)-\frac{1}{p}\langle J^{\prime}(u_n),u_n \rangle=\left(\frac{1}{p}-\frac{1}{p^{\ast}}\right)\|u_n\|_{p^{\ast}}^{p^{\ast}}-\lambda\left( \frac{1}{q}-\frac{1}{p}\right)\|u_n\|_q^q
\end{align*}
we consider the two cases $(i)$ and $(ii)$.

\vspace{0.3cm}
$\bullet$\, {\em Case (i).}
By \eqref{ter1.2}, \eqref{ter1.3} and \eqref{ter1.5}, we obtain
$$
c \geq\left(\frac{1}{p}-\frac{1}{p^{\ast}}\right)\left(\ell^{p^{\ast}}+\|u\|_{p^{\ast}}^{p^{\ast}}\right)\geq\left(\frac{1}{p}-\frac{1}{p^{\ast}}\right) \mathcal{S}^{\frac{p^{\ast}}{p^{\ast}-p}}
$$
which gives the desired contradiction. Hence, $\ell=0$ and by \eqref{ter1.3} we get $u_n\to u$ in $\mathrm{X}_0(\Omega)$.

\vspace{0.3cm}
$\bullet$\, {\em Case (ii).}
By \eqref{ter1.2}, \eqref{ter1.3}, H\"older and Young's inequalities
$$
\begin{aligned}
c &\geq\left(\frac{1}{p}-\frac{1}{p^{\ast}}\right)\left(\ell^{p^{\ast}}+\|u\|_{p^{\ast}}^{p^{\ast}}\right)-\lambda\left(\frac{1}{q}-\frac{1}{p}\right)\|u\|_q^q \\
&\geq\left(\frac{1}{p}-\frac{1}{p^{\ast}}\right)\left(\ell^{p^{\ast}}+\|u\|_{p^{\ast}}^{p^{\ast}}\right)-\lambda\left(\frac{1}{q}-\frac{1}{p}\right)|\Omega|^{\frac{p^{\ast}-q}{p^{\ast}}}\|u\|_{p^{\ast}}^q \\
&\geq\left(\frac{1}{p}-\frac{1}{p^{\ast}}\right)\left(\ell^{p^{\ast}}+\|u\|_{p^{\ast}}^{p^{\ast}}\right)-\left(\frac{1}{p}-\frac{1}{p^{\ast}}\right)\|u\|_{p^{\ast}}^{p^{\ast}}
-|\Omega|\left(\frac{1}{p}-\frac{1}{p^{\ast}}\right)^{-\frac{q}{p^{\ast}-q}}\left[\lambda\left(\frac{1}{q}-\frac{1}{p}\right)\right]^{\frac{p^{\ast}}{p^{\ast}-q}},
\end{aligned}
$$
which gives the desired contradiction using \eqref{ter1.5}. Hence, again \eqref{ter1.3} with $\ell=0$ yields $u_n\to u$ in $\mathrm{X}_0(\Omega)$, concluding the proof.

\end{proof}

\section{Sublinear case}\label{sec3}
In this section, we prove Theorem \ref{T3}.  For this, we assume along the section that $q\in(1,p)$ without further mentioning. Firstly, we recall the definition of genus and some its fundamental properties; see \cite{Rab} for more details.

Let $E$ be a Banach space and $A$ a subset of $E$. We say that $A$ is symmetric if $u\in A$ implies that $-u\in A$. For a closed symmetric set $A$ which does not contain the origin, we define the genus $\gamma(A)$ of $A$ as the smallest integer $k$ such that there exists an odd mapping $\phi\in C^0(A;\mathbb R^k\setminus \{0\})$. If there does not exist such a $k$, we put  $\gamma(A)=\infty$. Moreover, we set  $\gamma(\emptyset)=0$.

Then we have the following result.
\begin{proposition}\label{prop3.1}
Let $A$ and $B$ be closed symmetric subsets of $E$ which do not contain the origin. Then we have:
\begin{enumerate}
\item[$(i)$] if there exists an odd continuous mapping from $A$ to $B$, then $\gamma(A)\leq \gamma(B)$;
\item[$(ii)$] if there is an odd homeomorphism from $A$ onto $B$, then $\gamma(A)= \gamma(B)$;
\item[$(iii)$] if $\gamma(B)<\infty$, then $\gamma(A\setminus B)\geq \gamma(A)-\gamma(B)$;
\item[$(iv)$] the $k$-dimensional sphere $\mathbb{S}^{k}$ has a genus of $k+1$ by the Borsuk-Ulam Theorem;
\item[$(v)$] if $A$ is compact, then $\gamma(A)<\infty$ and there exist $\delta>0$ and a closed and symmetric neighborhood $N_{\delta}(A)=\{x\in E:\,\, \text{dist}(x,A)\leq \delta \}$ of $A$ such that $\gamma(N_{\delta}(A))=\gamma(A)$.
\end{enumerate}
\end{proposition}

We also need the following deformation lemma, as given in  \cite[Lemma 1.3]{AR}. For this, considering  $I\in C^{1}(E;\mathbb R)$, we set for any $c\in\mathbb R$
\begin{equation}\label{kappac}
K_c=\left\{u\in E:\,\,I^{\prime}(u)=0\mbox{ and }I(u)=c\right\}.
\end{equation}
Also, we use the following notation $I^d=\left\{u\in E:\,\, I(u)\leq d\right\}$ for $d\in\mathbb R$.

\begin{lemma}\label{pro1.0.31}
Let $E$ be an infinite-dimensional Banach space and let $I\in C^{1}(E;\mathbb R)$ be a functional satisfying the $(PS)_c$, with $c\in \mathbb{R}$. Let $U$ be any neighborhood of $K_c$. Then, there exists $\eta\in C^0([0,1]\times E; E)$, with  $\eta_t(x)=\eta (t,x)$, and constants $d_1$, $\varepsilon>0$ with $|c|>d_1$, such that:
\begin{itemize}
\item[$(i)$] $\eta_0= Id_E$;

\item[$(ii)$]$\eta_t(u)=u$  for any   $u\not\in I^{-1}([c-\varepsilon,c+\varepsilon])$  and  $t\in [0,1]$;

\item[$(iii)$] $\eta_t$ is a homeomorphism for any $t\in [0,1]$;

\item[$(iv)$] $I(\eta_t (u))\leq I(u)$ for any $u\in E$  and $t\in [0,1]$;

\item[$(v)$] $\eta_1 (I^{c+d_1}\setminus U)\subseteq I^{c-d_1}$;

\item[$(vi)$] if $K_c =\emptyset$, then $\eta_1 (I^{c+d_1})\subseteq I^{c-d_1}$;

\item[$(vii)$] if $I$ is even, then $\eta_t$ is odd for any $t\in [0,1]$.
\end{itemize}
\end{lemma}

Being $q\in(1,p)$, let us note that $J_\lambda$ is not bounded from below in $\mathrm{X}_0(\Omega)$.
Thus, we will apply a truncation argument, inspired by \cite{GA2}.
By \eqref{sob1} and Sobolev embedding, for any $u\in \mathrm{X}_0(\Omega)$
\begin{equation*}
    J_\lambda(u)\geq \displaystyle\frac{1}{p}\|u\|_{\mathrm{X}_0}^{ p}
		- \lambda\frac{\,C_q}{q} \|u\|_{\mathrm{X}_0}^q
		-\frac{\mathcal{S}^{p^{\ast}/p}}{p^{\ast}} \|u\|_{\mathrm{X}_0}^{p^{\ast}}=g_{\lambda}(\|u\|_{\mathrm{X}_0})
\end{equation*}
where
\begin{equation*}
g_\lambda(t)=\displaystyle\frac{1}{ p}t^{ p}-\frac{\lambda\, C_q}{q}  t^q-\frac{\mathcal{S}^{p^{\ast}/p}}{p^{\ast}} t^{p^{\ast}},\quad t\in [0,\infty).
\end{equation*}
Let us fix $R_1>0$ sufficiently small so that
$$\displaystyle\frac{1}{ p}R_1^{ p}-\frac{\mathcal{S}^{p^{\ast}/p}}{p^{\ast}} R_1^{p^{\ast}}>0,
$$
and define
\begin{equation}\label{lambda1}
\lambda_0=\frac{q}{2C_qR_1^q}\left(\displaystyle\frac{1}{ p}R_1^{ p}-\frac{\mathcal{S}^{p^{\ast}/p}}{p^{\ast}} R_1^{p^{\ast}}\right)>0,
\end{equation}
so that $g_{\lambda_0}(R_1)>0$.
Let us set
	\[R_0=\max\left\{t\in(0,R_1):\,\, g_{\lambda_0}(t)\leq0\right\}.
\]
Since  $q< p$, we see that  $g_{\lambda_0}(t)<0$ for $t$ near to $0$, and considering also $g_{\lambda_0}(R_1)>0$, we infer that  $g_{\lambda_0}(R_0)=0$.

Choose $\psi\in C_0^\infty\left([0,\infty)\right)$ such that $0\leq\psi(t)\leq1$, $\psi(t)=1$ for $t\in[0,R_0]$ and  $\psi(t)=0$ for $t\in[R_1,\infty)$.
Thus, we consider the truncated functional
\begin{equation*}
\widetilde{J}_\lambda(u)=\displaystyle\frac{1}{ p}\|u\|_{\mathrm{X}_0}^p- \displaystyle\frac{\lambda}{q} \|u\|_q^q-\displaystyle\frac{\psi(\|u\|_{\mathrm{X}_0}^{ p})}{p^{\ast}}\|u\|_{p^{\ast}}^{p^{\ast}}, \quad u\in \mathrm{X}_0(\Omega).
\end{equation*}
It immediately follows that $\widetilde{J}_\lambda(u)\to\infty$ as $\|u\|_{\mathrm{X}_0}\to\infty$. Hence, $\widetilde{J}_\lambda$ is
coercive and bounded from below.

\begin{lemma}\label{lem2.3.2}
There exists $\lambda_{\ast}>0$ such that, for any $\lambda\in(0,\lambda_{\ast})$ we have:
\begin{itemize}
    \item[$(i)$] if $\widetilde{J}_\lambda(u)\leq0$ then $\|u\|< R_0$, and $\widetilde{J}_\lambda(v)=J_\lambda(v)$ for any $v$ in a small neighborhood of $u$;
    \item[$(ii)$]   $\widetilde{J}_\lambda$ satisfies $(PS)_c$ condition for any $c<0$.
\end{itemize}
\end{lemma}
\begin{proof}
Let us set
\begin{equation}\label{lambda1stella}
\lambda_1^{\ast}=\frac{q}{pC_q}R_1^{p-q}.
\end{equation}
Also, let us choose $\lambda^{\ast}_2$ sufficiently small so that
\begin{equation}\label{lambda2stella}
    \frac{1}{N} \mathcal{S}^{\frac{N}{p}}-|\Omega|\left(\frac{1}{p}-\frac{1}{p^{\ast}}\right)^{-\frac{q}{p^{\ast}-q}}\left[\lambda^{\ast}_2\left(\frac{1}{q}-\frac{1}{p}\right)\right]^{\frac{p^{\ast}}{p^{\ast}-q}}>0.
\end{equation}
Thus, let us set $\lambda_{\ast}\leq\min\{\lambda_0,\lambda_1^{\ast},\lambda_2^{\ast}\}$, with $\lambda_0$ given in \eqref{lambda1}, and we fix $\lambda\in(0,\lambda_{\ast})$.

\vspace{0.3cm}
$\bullet$\, {\em Case $(i)$.}  Let us assume that $\widetilde{J}_\lambda(u)\leq0$. If  $\|u\|_{\mathrm{X}_0}\geq R_1$, by Sobolev embedding
\begin{equation*}
\widetilde{J}_\lambda(u)\geq\displaystyle\frac{1}{ p}\|u\|_{\mathrm{X}_0}^p-\displaystyle\frac{\lambda\,C_q}{q} \|u\|_{\mathrm{X}_0}^q=h_\lambda(\|u\|_{\mathrm{X}_0})
\end{equation*}
where the continuous function
\begin{equation*}
h_\lambda(t)=\displaystyle\frac{1}{ p}t^{ p}-\frac{\lambda\, C_q}{q}  t^q,\quad t\in [0,\infty),
\end{equation*}
admits just two roots $t_0=0$ and $t_{1,\lambda}=\left(\lambda\,pC_q/q\right)^{1/(p-q)}$, with $h_\lambda(t)\to\infty$ as $t\to\infty$ being $q<p$.
Since $R_1>t_{1,\lambda}$ thanks to $\lambda<\lambda_1^{\ast}$ given in \eqref{lambda1stella}, we get $0\geq\widetilde{J}_\lambda(u)>0$ the desired contradiction.

When $\|u\|_{\mathrm{X}_0}< R_1$, since  $0\leq\psi(t)\leq 1$ and $\lambda<\lambda_0$, we get
$$
0\geq\widetilde{J}_\lambda(u)\geq g_\lambda(\|u\|_{\mathrm{X}_0})\geq g_{\lambda_0}(\|u\|_{\mathrm{X}_0})
$$
which implies that $\|u\|<R_0$, by definition of $R_0$.
Furthermore, by continuity of  $\widetilde{J}_\lambda$ there exists a small neighborhood $U\subset B_{R_0}(0)$ of $u$ such that $\widetilde{J}_\lambda(v)<0$ for any $v\in  U$. Since,  $J_\lambda(v)=\widetilde{J}_\lambda(v)$ for any $v\in U\subset B_{R_0}(0)$, we conclude the proof of the first part.

\vspace{0.3cm}
$\bullet$\, {\em Case $(ii)$.}  Let $c<0$ and let $\{u_n\}_n\subset \mathrm{X}_0(\Omega)$ be a sequence satisfying
\[\widetilde{J}_\lambda(u_n)\rightarrow c\quad\mbox{and}\quad \widetilde{J}^{\prime}_\lambda(u_n)\rightarrow 0,\]
as $n\to \infty$.
Then, for any $n\in\mathbb N$ sufficiently large, we can assume $\widetilde{J}_\lambda(u_n)=J_\lambda(u_n)$ and $\widetilde{J}^{\prime}_\lambda(u_n)=J^{\prime}_\lambda(u_n)$, so that $\{u_n\}_n$ satisfies \eqref{palais-smale}.
Since $\widetilde{J}_\lambda$ is a coercive functional, then the sequence $\{u_n\}_n$ is bounded in $\mathrm{X}_0(\Omega)$. Thus, considering $\lambda<\lambda_2^{\ast}$ given in \eqref{lambda2stella}, we can apply Lemma \ref{lemmaPS} and we conclude the proof.
\end{proof}

 \begin{lemma}\label{lem2.3.3}
Let $\lambda>0$ and let $k\in \mathbb{N}$. Then, there exists $\varepsilon=\varepsilon(\lambda,k)>0$ such that $\gamma(\widetilde{J}_\lambda^{-\varepsilon})\geq k$, where  $\widetilde{J}_\lambda^{-\varepsilon}= \big\{u\in \mathrm{X}_0(\Omega):\,\, \widetilde{J}_\lambda(u)\leq-\varepsilon\big\}$.
 \end{lemma}
   \begin{proof}
   Fix $\lambda>0$  and $k\in \mathbb{N}$. We can consider  $\mathcal{V}_k$ a $k$-dimensional vectorial subspace of $\mathrm{X}_0(\Omega)$. Then $\|\cdot\|_{\mathrm{X}_0}$, $\|\cdot\|_q$ are equivalent in $\mathcal{V}_k$, so that there exists $C(k)>0$ such that
   \begin{equation}\label{fi3}
      C(k)\|u\|_{\mathrm{X}_0}^q\leq \|u\|_q^q \quad\mbox{for any }u\in \mathcal{V}_k.
   \end{equation}
By \eqref{fi3}, for any $u\in \mathcal{V}_k$ with $\|u\|_{\mathrm{X}_0}\leq R_0$, we get
\begin{equation}\label{tempora}
\widetilde{J}_\lambda(u)=J_\lambda(u)\leq\displaystyle\frac{1}{ p}\|u\|_{\mathrm{X}_0}^p-\displaystyle\frac{\lambda\,C(k)}{q} \|u\|_{\mathrm{X}_0}^q.
\end{equation}
Now, let $\varrho>0$ such that
\begin{equation}\label{raggi}
\varrho<\min\left\{R_0, \left[\frac{\lambda\,C(k)\, p}{q}\right]^{\frac{1}{ p-q}}\right\}
\end{equation}
\noindent
and let $S_\varrho=\big\{u\in \mathcal{V}_k:\,\, \|u\|_{\mathrm{X}_0}= \varrho\big\}$. Clearly $S_\varrho$ is homeomorphic to the  $(k-1)$-dimensional sphere $\mathbb S^{k-1}$.

Furthermore, for any  $u\in S_\varrho$  by \eqref{tempora}
\begin{equation*}
\begin{alignedat}2
\widetilde{J}_\lambda(u)\leq\varrho^q\left(\frac{1}{ p}\varrho^{ p-q}-\frac{\lambda\,C(k)}{q}\right)<0
\end{alignedat}
\end{equation*}
where the last inequality follows by \eqref{raggi}.
Hence, we can find a constant $\varepsilon=\varepsilon(\lambda,k)>0$ such that $\widetilde{J}_\lambda(u)<-\varepsilon$ for any $u\in S_\varrho$.
As a consequence, we have $S_\varrho\subset \widetilde{J}_{\lambda}^{-\varepsilon}\setminus\{0\}$. By parts $(ii)$ and $(iv)$ of Proposition \ref{prop3.1}  we get
$$\gamma(\widetilde{J}_{\lambda}^{-\varepsilon})\geq \gamma(S_\varrho)= k,$$
concluding the proof.
\end{proof}

Let us set the numbers
\begin{equation}\label{ck}
c_k=\displaystyle\inf_{A\in \Sigma_k } \sup_{u\in A} \widetilde{J}_\lambda(u),
\end{equation}
with
$$
\Sigma_k=\left\{A\subset \mathrm{X}_0(\Omega):\,\, A \mbox{ is closed,\, symmetric such that }\,0\not\in A\, \mbox{ and }\,\gamma(A)\geq k \right\}.
$$
Clearly $c_k\leq c_{k+1}$ for any $k\in \mathbb{N}$.

\begin{lemma}\label{negative}
Let $\lambda>0$ and let $k\in\mathbb N$. Then, $c_k<0$.
\end{lemma}
\begin{proof}
Fix $\lambda>0$  and $k\in\mathbb N$.
By Lemma \ref{lem2.3.3}, there exists $\varepsilon>0$ such that $\gamma(\widetilde{J}_\lambda^{-\varepsilon})\geq k$. Furthermore,  $\widetilde{J}_\lambda^{-\varepsilon}\in\Sigma_k$ since $\widetilde{J}_\lambda$ is continuous and even. Taking into account that $\widetilde{J}_\lambda(0)=0$, we have $0\not\in\widetilde{J}_\lambda^{-\varepsilon}$ and  $\displaystyle \sup_{u\in \widetilde{J}_\lambda^{-\varepsilon}}\widetilde{J}_\lambda(u)\leq-\varepsilon$. Therefore, remembering also that $\widetilde{J}$ is bounded from below, we conclude
	\[-\infty<c_k=\inf_{A\in\Sigma_k}\sup_{u\in A}\widetilde{J}_\lambda(u)\leq\sup_{u\in \widetilde{J}_\lambda^{-\varepsilon}}\widetilde{J}_\lambda(u)\leq-\varepsilon<0.
\]
\end{proof}

Now, we state the following technical lemma, with
$$
K_c=\left\{u\in \mathrm{X}_0(\Omega):\,\,\widetilde{J}^{\prime}_\lambda(u)=0\,\mbox{ and }\,\widetilde{J}_\lambda(u)=c\right\}.
$$

\begin{lemma}\label{lemma3.4}
Let $\lambda\in(0,\lambda_{\ast})$, with $\lambda_{\ast}$ as in Lemma \ref{lem2.3.2}, and let $k\in\mathbb{N}$. If $c=c_k=c_{k+1}=\cdots =c_{k+m}$, for some $m\in\mathbb{N}$, then
$$\gamma(K_c)\geq m+1.$$
\end{lemma}

\begin{proof}
Fix $\lambda\in(0,\lambda_{\ast})$ and $k\in\mathbb{N}$.
By Lemma \ref{negative} we have $c=c_k=\cdots =c_{k+m}<0$, so we can apply part $(ii)$ of Lemma \ref{lem2.3.2} to see
that $K_c$ is compact.

Let us assume by contradiction that $\gamma(K_c)\leq m$.
Then, by part $(v)$ of Proposition \ref{prop3.1} there exists $\delta > 0$ such that
$\gamma(N_\delta(K_c))=\gamma(K_c)\leq m$ where $N_\delta(K_c)$ is given as
\[
N_\delta(K_c)=\left\{v\in \mathrm{X}_0(\Omega):\,\, \text{dist}(v,K_c)\leq \delta\right\}.
\]
Thus, by Proposition \ref{pro1.0.31} there exists $\varepsilon\in(0,-c)$ and an odd homeomorphism
$\eta :\mathrm{X}_0(\Omega)\to \mathrm{X}_0(\Omega)$ such that
\begin{equation}\label{e3.5}
\eta(\widetilde{J}_\lambda^{c+\varepsilon}\setminus
N_\delta(K_{c})) \subset
\widetilde{J}_\lambda^{c-\epsilon}.
\end{equation}
By definition of $c=c_{k+m}$ in \eqref{ck}, there exists $A \in
\Sigma_{k+m}$ such that
$$
\displaystyle\sup_{u\in A}\widetilde{J}_\lambda(u) <
c+ \varepsilon,
$$
that is $A\subset\widetilde{J}_\lambda^{c+\varepsilon}$ and so by \eqref{e3.5}
\begin{equation}\label{e3.8}
\eta(A\setminus N_\delta(K_{c}))\subset \eta(\widetilde{J}_\lambda^{c+\epsilon}\setminus
N_\delta(K_{c})) \subset \widetilde{J}_\lambda^{c-\epsilon}.
\end{equation}
On the other hand, by parts $(i)$ and $(iii)$ of Proposition \ref{prop3.1}, we get
$$
\gamma(\eta(\overline{A\setminus N_\delta(K_{c})})) \geq
\gamma(\overline{A\setminus N_\delta(K_{c})}) \geq
\gamma(A) - \gamma(N_\delta(K_{c}))\geq k.
$$
Hence, $\eta(\overline{A\setminus N_\delta(K_{c})}) \in \Sigma_k$, and so
$$
\sup_{u\in \eta(\overline{A\setminus
N_\delta(K_{c})})}\widetilde{J}_\lambda(u) \geq c_k=c
$$
which contradicts \eqref{e3.8}. Thus, we conclude $\gamma(K_c)\geq m+1$.
\end{proof}

Finally, we are in a position to prove Theorem \ref{T3}.

\begin{proof}[Proof of Theorem \ref{T3}]
Let $\lambda\in(0,\lambda_{\ast})$ with $\lambda_{\ast}$ given in Lemma \ref{lem2.3.2}.
By Lemma \ref{negative} we have $c_k<0$.  Hence, from part $(ii)$ of Lemma \ref{lem2.3.2} we know that the functional $\widetilde{J}_\lambda$ satisfies the $(PS)_{c_k}$ condition. Thus, by a standard argument, see \cite{Rab} for example, $c_k$ is a critical value of $\widetilde{J}_\lambda$ for any $k\in\mathbb N$.

We distinguish two situations.
If $-\infty<c_1<c_2<\ldots<c_k<c_{k+1}<\ldots$, then $\widetilde{J}_\lambda$ admits infinitely many critical values.
If there exist $k$, $m\in\mathbb N$ such that $c_{k}=c_{k+1}=\cdots c_{k+m}=c$, then $\gamma(K_c)\geq m+1\geq 2$ by Lemma \ref{lemma3.4}. Thus, the set $K_c$ has infinitely many points, see  \cite[Remark 7.3]{Rab}, which are infinitely many critical values for $\widetilde{J}_\lambda$ by part $(ii)$ of Lemma \ref{lem2.3.2}.

Then, by part $(i)$ of Lemma \ref{lem2.3.2} we get infinitely many negative critical values for $J_\lambda=\widetilde{J}_\lambda$ and so problem \eqref{P} has infinitely many weak solutions.
\end{proof}

%%%%%%%%%%%%%%%%%%%%%%%%%%%%%%%%%%%%%%%%%%%%%%%%%%%%%%%%%%%%%%%%%%%%%%%%%%%%%%%%%%%%%%%%%%%%%%%%%%%%%%%%%%%%%%%%%%%%%%%%%%%

\section{Linear case}\label{sec4}

In this section we prove Theorem \ref{T2}, studying \eqref{P} with $q=p$. For this, we exploit the \textit{cohomological index theory}, introduced by Fadell and Rabinowitz, as set in \cite[Chapter 2]{morse}.

Let $E$ be a Banach space and $\mathcal{B}$ denotes the class of symmetric subsets of $E\setminus\{0\}$. For any  $X\in\mathcal{B}$,  let $\overline{X}=X/\mathbb{Z}_2$ be the quotient space of $\mathcal B$ with each $u$ and $-u$ identified. Define $g:\overline{X}\to\mathbb{R}P^\infty$ to be the classifying map of $\overline{X}$, and
let $g^{\ast}:H^{\ast}(\mathbb{R}P^\infty)\longrightarrow H^{\ast}(\overline{X})$ be the induced homomorphism of the Alexander-Spanier cohomology rings.
The $\mathbb{Z}_2$-cohomological index of $X$ is defined by
$$
i(X):=\left\{
\begin{matrix}
&\hspace{-0.35cm}\sup\left\{k\in\mathbb{N}:\,\, g^{\ast}(\omega^{k-1})\neq 0\right\},& {\mbox{ if } X\neq\emptyset,}\\
&\hspace{-4.7cm}0,& {\mbox{ if } X=\emptyset,}
\end{matrix}
\right.
$$
where $\omega\in H^1(\mathbb{R}P^{\infty})$ is the generator of the polynomial ring $H^{\ast}(\mathbb{R}P^\infty)=\mathbb{Z}_2[\omega]$.

We recall some basic properties of the \textit{cohomological index}, provided in \cite[Propositions 2.12 and 2.14]{morse}.
\begin{proposition}\label{pro1.0.21}
The index $i:\mathcal{B}\longrightarrow \mathbb{N}\displaystyle\cup \{0,\infty\}$ satisfies the following properties:
\begin{itemize}
    \item[$(i)$] if $h:X\rightarrow Y $ is an equivariant map (in particular if $X\subset Y$), then
    $$i(X)\leq i(Y)$$
and equality holds when $h$ is an equivariant homeomorphism;
\item[$(ii)$] if $W$ is a normed linear space and  $X\subset W$ is symmetric and does not contain the origin, then
$$i(X)\leq \text{Dim}(W);$$
\item[$(iii)$] if $U$ is s a bounded closed symmetric neighborhood of the origin in a normed linear space $W$, then
$$i(\partial U)= \text{Dim}(W);$$
\item[$(iv)$] if $X$ is the disjoint union of a pair of subsets $U$, $-U$, then $i(X)=1$. In particular, $i(X)=1$ when $X$ is a finite set;
\item[$(v)$] if $X$ is compact, then $i(X)<\infty$.
\end{itemize}
\end{proposition}

Now, we give an abstract critical point result for an even functional $\Phi\in C^1(E;\mathbb R)$.
Let $r>0$ and let $S_r=\left\{u\in E:\,\,\|u\|=r\right\}$. Let $\Gamma$ denote the group of odd
homeomorphisms of $E$ that are the identity outside $\Phi^{-1}(0,d)$ for $0<d\leq \infty$. The pseudo-index of $X\in \mathcal{B}$  related to $i$, $S_r$ and $\Gamma$ is set as
$$i_r^{\ast}(X):=\min_{\gamma\in \Gamma}i(\gamma(X)\cap S_r).
$$
The following critical point result is given in \cite[Theorem 2.2]{PSY}.

\begin{proposition}\label{pro1.0.23}
Let $A_0$, $B_0$ be symmetric subsets of $S_1$ such that $A_0$ is compact in $E$, $B_0$ is closed in $E$, and
$$i(A_0)\geq k+m,\qquad i(S_1\setminus B_0)\leq k$$
for some integers $k\geq 0$ and $m\geq 1$. Assume there exist $b\in \mathbb{R}$ and $R>r>0$ such that
$$\sup_{u\in A} \Phi(u)\leq 0< \inf_{u\in B} \Phi(u),\qquad \sup_{u\in X} \Phi(u) <b, $$
where  $A=\{Ru:\, u\in A_0\}$, $B=\{ru:\,u\in B_0\}$ and $X=\{tu:\, u\in A, t\in [0,1]\}$.
For $j=k+1,\ldots, k+m$ let
$$\mathcal{B}^{\ast}_j=\{M\in \mathcal{B}:\,\, M\mbox{ is compact and }\,i_r^{\ast}(M)\geq j \}$$
and let
$$c_j^{\ast}:=\inf_{M\in \mathcal{B}^{\ast}_j} \sup_{u\in M} \Phi(u).$$
Then,
$$
\inf_{u\in B} \Phi(u)\leq c_{k+1}^{\ast}\leq \ldots\leq c_{k+m}^{\ast}\leq \sup_{u\in X} \Phi(u),$$
in particular $0<c_j^{\ast}<b$. If, in addition, $\Phi$ satisfies the $(PS)_c$ condition for any $c\in(0,b)$, then any $c_j^{\ast}$ is a critical value for $\Phi$ and there are $m$ distinct pairs of associated critical points.
\end{proposition}

Let us set $A_p$, $B_p\in C^0\left(\mathrm{X}_0(\Omega); (\mathrm{X}_0(\Omega))^{\ast}\right)$ as
\begin{equation}\label{defAp}
\langle A_p(u),v\rangle=p\left[\int_{\Omega} |\nabla u|^{p-2}\nabla u\cdot \nabla v \,dx+\iint_{\mathbb{R}^{2N}}\mathcal{A}u(x,y)[v(x)-v(y)]\,dx\,dy  \right]
\end{equation}
and
\begin{equation}\label{defBp}
\langle B_p(u),v\rangle= p\int_{\Omega}|u|^{ p-2}uv \,dx
\end{equation}
for any $u$, $v\in \mathrm{X}_0(\Omega)$, respectively.
Inspired by \cite[Chapter 1]{morse}, we have the following properties.

\begin{lemma}\label{Ap}
Let $p>1$. Then, we have:
\begin{itemize}
\item[$(A_1)$] $A_p$ is $(p-1)$-homogenous and odd;
\item[$(A_2)$] $A_p$ uniformly positive, that is there exists $\mathrm{c}_0>0$ such that
$$\langle A_p(u),u\rangle\geq \mathrm{c}_0\|u\|_{\mathrm{X}_0}^{ p}\quad \mbox{for any }u\in \mathrm{X}_0(\Omega);$$
\item[$(A_3)$] $A_p$ is a potential operator;
\item[$(A_4)$] $A_p$ is a mapping of $(S)$, that is any sequence $\{u_n\}_n\subset  \mathrm{X}_0(\Omega)$ such that
\begin{equation}\label{tipes}
u_n\rightharpoonup  u,\quad\text{and}\quad \langle A_p(u_n),u_n-u\rangle\to 0,\quad\text{as } n\to \infty
\end{equation}
has a subsequence that converges strongly to $u\in \mathrm{X}_0(\Omega)$;
\item[$(A_5)$] $A_p$ is strictly monotone, that is
\begin{equation}\label{mono}
\langle A_p(u)-A_p(v),u-v \rangle>0\quad \mbox{for any } u,\,v\in \mathrm{X}_0(\Omega),\mbox{ with } u\neq v,
\end{equation}
and the equality holds if and only if $u=v$.
\end{itemize}
\end{lemma}

\begin{proof}
Properties $(A_1)$-$(A_3)$ directly follow from definition \eqref{defAp}.

\vspace{0.3cm}
$\bullet$\, {\em Property $(A_4)$.} Let $\{u_n\}_n\subset \mathrm{X}_0(\Omega)$ be a sequence satisfying \eqref{tipes}.
Then, $\{u_n\}_n$ is bounded in $\mathrm{X}_0(\Omega)$ and so, up to a subsequence,
\begin{equation}\label{seg1.1}
 u_n\rightarrow u\mbox{ in } L^r(\Omega),\, \,r\in[1,p^{\ast}),\qquad u_n\to u\text{ a.e. in } \Omega.
\end{equation}
By \eqref{tipes}, as $n\to\infty$ we have
\begin{align}\label{lamb1}
\text{o}_n (1)= & \left\langle A_p\left(u_n\right),u_n-u\right\rangle \nonumber \\
= &\,p\left[ \int_{\Omega}\left|\nabla u_n\right|^{p-2} \nabla u_n\cdot \nabla \left(u_n-u\right) dx
+\iint_{\mathbb{R}^{2N}}\mathcal{A}u_n(x,y)[(u_n-u)(x)-(u_n-u)(y)]\,dx\,dy\right].
\end{align}
On the other side, still by \eqref{tipes}
\begin{equation}\label{equaly11}
  \int_{\Omega}|\nabla u_n|^{p-2}\nabla u_n\cdot\nabla(u_n-u)dx= \int_{\Omega} \left(|\nabla u_n|^{p-2}\nabla u_n-|\nabla u|^{p-2}\nabla u    \right)\cdot\nabla(u_n-u)dx+\text{o}_n(1),
\end{equation}
and denoting $v_n=u_n-u$
\begin{equation}
\begin{aligned}\label{equality2}
\iint_{\mathbb{R}^{2N}}\mathcal{A}u_n(x,y)[v_n(x)-v_n(y)]\,dx\,dy=&\,\iint_{\mathbb{R}^{2N} }\left[ \mathcal{A}u_n(x,y)-\mathcal{A}u(x,y)\right][v_n(x)-v_n(y)]\,dx\,dy
\\
   &+\iint_{\mathbb{R}^{2N}} \mathcal{A}u(x,y)[v_n(x)-v_n(y)]\,dx\,dy
   \\
=&\,\iint_{\mathbb{R}^{2N} }\left[ \mathcal{A}u_n(x,y)-\mathcal{A}u(x,y)\right][v_n(x)-v_n(y)]\,dx\,dy +\text{o}_n(1),
\end{aligned}
\end{equation}
as $n\to\infty$.
By the well-known Simon's inequalities, see \cite[formula (2.2)]{Simon-1978}, the right-hand sides of \eqref{equaly11} and of \eqref{equality2} are nonnegative, so by \eqref{lamb1} we get in particular
\begin{equation*}
    \lim_{n\to\infty} \int_{\Omega} \left(|\nabla u_n|^{p-2}\nabla u_n-|\nabla u|^{p-2}\nabla u    \right)\cdot\nabla(u_n-u)dx=0.
\end{equation*}
Using again \cite[formula (2.2)]{Simon-1978} and Lemma \ref{lemmaBDD}, we have
\begin{align*}
  \|u_n-u\|_{\mathrm{X}_0(\Omega)}^p&\leq (\mathrm{c}+1)\int_{\Omega}|\nabla u_n-\nabla u|^pdx\\
	&\leq (\mathrm{c}+1) \int_{\Omega} \left(|\nabla u_n|^{p-2}\nabla u_n-|\nabla u|^{p-2}\nabla u    \right)\cdot\nabla(u_n-u)dx\to 0
\end{align*}
as $n\to \infty$. Thus $u_n\to u$ in $\mathrm{X}_0(\Omega)$, concluding the proof.

\vspace{0.3cm}
$\bullet$\, {\em Property $(A_5)$.} By H\"older's inequality, for any $u$, $v\in \mathrm{X}_0(\Omega)$ we have
\begin{equation}\label{dl1}
\begin{aligned}
\langle A_p(u),v\rangle&=p\left[\int_{\Omega} |\nabla u|^{p-2}\nabla u\cdot\nabla v\,dx +\iint_{\mathbb{R}^{2N}}\mathcal{A}u(x,y)(v(x)-v(y))\,dx\,dy\right]\\
&\leq p\left(\|\nabla u\|_p^{p-1}\|\nabla v\|_p+[u]_{s,p}^{p-1}[v]_{s,p}\right)\\
&\leq p\|u\|_{\mathrm{X}_0}^{p-1}\|v\|_{\mathrm{X}_0}
\end{aligned}
\end{equation}
Hence, we get
\begin{align}\label{tenq0}
 \langle A_p(u)-A_p(v), u-v \rangle &=p\|u\|_{\mathrm{X}_0}^p-\langle A_p(u),v\rangle-\langle A_p(v),u\rangle +p\|v\|_{\mathrm{X}_0}^p
 \nonumber\\[0.35cm]
 &\geq p\left(\|u\|_{\mathrm{X}_0}^p-\|u\|_{\mathrm{X}_0}^{ p-1}\|v\|_{\mathrm{X}_0}-\|v\|_{\mathrm{X}_0}^{ p-1}\|v\|_{\mathrm{X}_0}+|u\|_{\mathrm{X}_0}^p\right)
 \nonumber\\[0.35cm]
 &=p\left(\|u\|_{\mathrm{X}_0}^{ p-1}-\|v\|_{\mathrm{X}_0}^{ p-1}\right)(\|u\|_{\mathrm{X}_0}-\|v\|_{\mathrm{X}_0}) \nonumber\\[0.35cm]
 &\geq 0.
\end{align}
Thus, if $\langle A_p(u)-A_p(v), u-v \rangle =0$, then necessarily $\|u\|_{\mathrm{X}_0}=\|v\|_{\mathrm{X}_0}$.

By \eqref{dl1} we have
\begin{equation}
\begin{aligned}\label{maior0}
    \langle A_p(u),u\rangle&=p\|u\|_{\mathrm{X}_0}^{ p}=p\|u\|_{\mathrm{X}_0}^{ p-1}\|v\|_{\mathrm{X}_0}\geq \langle A_p(u),v\rangle\quad\Longrightarrow \quad \langle A_p(u),u-v\rangle\geq 0,
    \\[0.35cm]
      \langle A_p(v),v\rangle&=p\|v\|_{\mathrm{X}_0}^{ p}=p\|v\|_{\mathrm{X}_0}^{ p-1}\|u\|_{\mathrm{X}_0} \geq \langle A_p(v),u\rangle\quad\Longrightarrow \quad \langle A_p(v),v-u\rangle\geq 0.
\end{aligned}
\end{equation}
If $\langle A_p(u)-A_p(v), u-v \rangle =0$, then
$$\langle A_p(u)-A_p(v),u-v \rangle=\langle A_p(u),u-v \rangle+\langle A_p(v),v-u \rangle=0, $$
and so thanks to \eqref{maior0} we get
$$\langle A_p(u),u-v\rangle= 0\quad\mbox{and}\quad\langle A_p(v),v-u\rangle= 0,$$
which yield in particular, since also $\|u\|_{\mathrm{X}_0}=\|v\|_{\mathrm{X}_0}$, that
\begin{equation}
\begin{aligned}\label{identl}
 \nonumber\langle A_p(u),v\rangle&=p\left[\int_{\Omega} |\nabla u|^{p-2}\nabla u\cdot\nabla v\,dx+\iint_{\mathbb{R}^{2N}}\mathcal{A}u(x,y)(v(x)-v(y))\,dx\,dy\right]
 \\
 &=p\left[\int_{\Omega} |\nabla u|^{p-1}|\nabla v|\,dx +\iint_{\mathbb{R}^{2N}}|\mathcal{A}u(x,y)||v(x)-v(y)|\,dx\,dy\right]\\
 &= p\|u\|_{\mathrm{X}_0}^{ p-1}\|v\|_{\mathrm{X}_0}.
\end{aligned}
\end{equation}
%\begin{equation}
%\begin{aligned}\label{identl}
 %\nonumber\langle A_p(u),v\rangle&=p\left(\int_{\Omega} |\nabla u|^{p-2}\nabla u\cdot\nabla v\,dx+\iint_{\mathbb{R}^{2N}}\mathcal{A}u(x,y)(v(x)-v(y))\,dx\,dy\right)
 %\\
 %&=p\left(\int_{\Omega} |\nabla u|^{p-1}|\nabla v|\,dx +\iint_{\mathbb{R}^{2N}}|\mathcal{A}u(x,y)||v(x)-v(y)|\,dx\,dy\right)= p\|u\|^{ p-1}\|v\|,
 %\\[0.35cm]
%\nonumber \langle A_p(v),u\rangle&=p\left(\int_{\Omega} |\nabla v|^{p-2}\nabla v\cdot\nabla u\,dx+\iint_{\mathbb{R}^{2N}}\mathcal{A}v(x,y)(u(x)-u(y))\,dx\,dy\right)
 %\\
 %&=p\left(\int_{\Omega} |\nabla v|^{p-1}|\nabla u|\,dx +\iint_{\mathbb{R}^{2N}}|\mathcal{A}v(x,y)||u(x)-u(y)|\,dx\,dy\right)= p\|v\|^{ p-1}\|u\|.
%\end{aligned}
%\end{equation}
From this, we get
\begin{equation}
\begin{aligned}\label{paralel}
0&\displaystyle \geq \int_{\Omega} \left( |\nabla u|^{p-2}\nabla u\cdot\nabla v-|\nabla u|^{p-1}|\nabla v| \right)dx
\\
& \displaystyle =\iint_{\mathbb{R}^{2N}}\left[\,|\mathcal{A}u(x,y)||v(x)-v(y)|- \mathcal{A}u(x,y)(v(x)-v(y))\, \right]\,dx\,dy\geq 0.
\end{aligned}
\end{equation}
Hence, we deduce that
$$\int_{\Omega} \left( |\nabla u|^{p-2}\nabla u\cdot\nabla v-|\nabla u|^{p-1}|\nabla v| \right)dx=0,$$
and so there exists $c(x)\geq 0$ such that $\nabla u(x)=c(x)\nabla v(x)$ for a.e. $x\in \Omega$. On the other hand, by equality \eqref{identl} used in \eqref{dl1}, we have $a_1-a_2=b_2-b_1$, where
\begin{align*}
   a_1& \displaystyle =\int_{\Omega}|\nabla u|^{p-1}|\nabla v| dx,\qquad\qquad\qquad\quad\,\,\,\, a_2=\|\nabla u\|_p^{p-1}\|\nabla v\|_p,
\\
b_1&\displaystyle =\iint_{\mathbb{R}^{2N}}\mathcal{A}u(x,y)(v(x)-v(y))\,dx\,dy,\qquad b_2=[u]_{s,p}^{p-1}[v]_{s,p}.
\end{align*}
By H\"older's inequality $a_1-a_2\leq 0$ and $b_2-b_1\geq 0$, so that $a_1=a_2$ which implies $|\nabla u(x)|=\alpha|\nabla v(x)|$ for a.e. $x\in \Omega$, with
a suitable $\alpha\geq0$. Thus, we get $u(x)=\alpha v(x)$ for a.e. $x\in \Omega$, but remembering $\|u\|_{\mathrm{X}_0}=\|v\|_{\mathrm{X}_0}$, if $\|v\|_{\mathrm{X}_0}\neq 0$ then $\alpha=1$ and $u=v$. If $\|v\|=0$ then $\|u\|_{\mathrm{X}_0}=0$ and so $u=v=0$.
\end{proof}

\begin{lemma}\label{Bp}
Let $p>1$. Then, we have:
\begin{itemize}
\item[$(B_1)$] $B_p$ is $(p-1)$-homogeneous and odd;
\item[$(B_2)$] $B_p$ is strictly positive, that is
$$\langle B_p(u),u\rangle>0\quad\mbox{for any } u\neq 0;$$
\item[$(B_3)$] $B_p$ is a compact potential operator.
\end{itemize}
\end{lemma}

\begin{proof}
By definition  \eqref{defBp} we clearly have $B_p$ satisfies $(B_1)$-$(B_2)$.
In order to prove that $B_p$ is a compact operator, let  $\{u_n\}_n\subset \mathrm{X}_0(\Omega)$ and $u\in \mathrm{X}_0(\Omega)$ such that
$$u_n\rightharpoonup  u\quad\text{in } \mathrm{X}_0(\Omega).$$
We show that  $B_p(u_n)\to B_p(u)$ in $(\mathrm{X}_0(\Omega))^{\ast}$  as $n\to \infty$. Denoting by $\|\cdot\|_{(\mathrm{X}_0(\Omega))^{\ast}}$ the norm in the dual space $(\mathrm{X}_0(\Omega))^{\ast}$, then
$$
\begin{aligned}
 \|B_p(u_n)-B_p(u)\|_{(\mathrm{X}_0(\Omega))^{\ast}}&=\sup_{v\in \mathrm{X}_0(\Omega),\, \|v\|_{\mathrm{X}_0}\leq 1}\left|B_p(u_n)(v)-B_p(u)(v)\right|
     \\
   &=\sup_{v\in \mathrm{X}_0(\Omega),\, \|v\|_{\mathrm{X}_0}\leq 1}  p\left |\int_{\Omega}(|u_n|^{ p-2}u_n-|u|^{p-2}u)v \,dx \right|
\end{aligned}
$$
By H\"older's inequality and Sobolev embedding
$$
\begin{aligned}
     \|B_p(u_n)-B_p(u)\|_{(\mathrm{X}_0(\Omega))^{\ast}}&\leq  \sup_{v\in \mathrm{X}_0(\Omega),\, \|v\|_{\mathrm{X}_0}\leq 1} p \left(\int_{\Omega} \left||u_n|^{ p-2}u_n-|u|^{ p-2}u\right|^{\frac{ p}{ p-1}}\,dx\right)^{\frac{ p-1}{ p}}\|v\|_p\\
     &\leq C\left(\int_{\Omega}\left||u_n|^{ p-2}u_n-|u|^{p-2}u\right|^{\frac{ p}{ p-1}}\,dx\right)^{\frac{ p-1}{ p}},
\end{aligned}
$$
for a suitable $C>0$.
From this, by the Lebesgue Dominated Convergence Theorem, up to a subsequence, we conclude that $B_p(u_n)\to B_p(u)$ em $(\mathrm{X}_0(\Omega))^{\ast}$, as $n\to \infty$.

\end{proof}

Let us set the potentials of $A_p$ and $B_p$ as
$$\mathcal{I}_{p}(u)=\|u\|_{\mathrm{X}_0}^{ p}\quad\text{ and }\quad  \mathcal{J}_{p}(u)=\|u\|_{ p}^{p},$$
respectively. Let us set
$$\Psi(u)=\displaystyle\frac{1}{\mathcal{J}_{p}(u)},\qquad u\in \mathrm{X}_0(\Omega)\setminus \{0\},$$
and we consider its restriction $\widetilde{\Psi}$ as
\begin{equation}\label{psitil}
\widetilde{\Psi}=\Psi|_{\mathcal{N}},\quad\text{with }\mathcal{N}=\{u\in \mathrm{X}_0(\Omega):\,\,\mathcal{I}_{p}(u)=1 \}.
\end{equation}
It is well known that $\mathcal{N}$ is a complete manifold of class $C^1$, whenever $p\geq 2$.

Let $\mathcal{F}$ denote the class of symmetric subsets of $\mathcal{N}$ in $\mathrm{X}_0(\Omega)$. For any $k\in \mathbb{N}$ we set
\begin{equation}\label{autperera}
\lambda_k=\inf_{M\in\mathcal{F}_k}\sup_{u\in M} \widetilde{\Psi}(u)\quad\text{where}\quad\mathcal{F}_k=\{M\in\mathcal{F}:\,\, i(M)\geq k\}.
\end{equation}
Then, remembering the usual notation for sublevel sets and superlevel sets, as
$$
\widetilde{\Psi}^a=\left\{u\in\mathcal N:\,\,\widetilde{\Psi}(u)\leq a\right\},\qquad
\widetilde{\Psi}_a=\left\{u\in\mathcal N:\,\,\widetilde{\Psi}(u)\geq a\right\},\qquad a\in\mathbb R,
$$
we can prove the following properties for $\{\lambda_k\}_k$.

\begin{lemma}\label{ein2}
We have that $\{\lambda_k\}_k$ is a nondecreasing sequence of eigenvalues for \eqref{kir1}, satisfying the following properties:
\begin{itemize}
\item[$(i)$] if $\lambda_k=\ldots=\lambda_{k+m-1}=\lambda$, then $i(E_\lambda)\geq m$, where
$$E_\lambda=\{u\in\mathcal N:\,\,\widetilde{\Psi}^{\prime}(u)=0\,\mbox{ and }\,\widetilde{\Psi}(u)=\lambda\};
$$
\item[$(ii)$] the first eigenvalue is set as
$$\lambda_1=\min_{u\in \mathcal{N}}\widetilde{\Psi}(u)=\min_{u\neq 0}\displaystyle\frac{\mathcal{I}_{p}(u)}{\mathcal{J}_{p}(u)};
$$
\item[$(iii)$]  we have $i(\mathcal{N}\setminus \widetilde{\Psi}_{\lambda_k})<k\leq i(\widetilde{\Psi}^{\lambda_k})$.
If $\lambda_k<\lambda<\lambda_{k+1}$, then
$$i(\widetilde{\Psi}^{\lambda_k})=i(\mathcal{N}\,\setminus \widetilde{\Psi}_{\lambda_k})=i(\widetilde{\Psi}^{\lambda})=i(\mathcal{N}\,\setminus \widetilde{\Psi}_{\lambda_{k+1}})=k;$$
\item[$(iv)$] $\lambda_k\to \infty$, as $k\to \infty.$
\end{itemize}
\end{lemma}
\begin{proof}
Thanks to Lemmas \ref{Ap} and \ref{Bp}, we can apply \cite[Theorem 4.6]{morse}.
\end{proof}

The study of eigenvalues of problem \eqref{kir1} is not new. We can refer to \cite{BDD}, \cite{BDVV23} and \cite{GU}. However, in Lemma \ref{ein2} we provide more information of $\{\lambda_k\}_k$, strictly related to the cohomological index theory, completing the picture of \cite{GU}.

Let us observe that $\lambda_1>\frac{\mathcal{S}}{|\Omega|^{p/N}}$, where $\mathcal{S}$ is set in \eqref{sob1}. Indeed, if  $u_1\in\mathcal{N}$ an eigenfunction associated with $\lambda_1$, then by H\"older's inequality
\begin{equation}\label{lan2}
 \lambda_1=\displaystyle\frac{\|\nabla u_1\|_p^{ p}+[u_1]_{s,p}^p}{\|u_1\|_{ p}^{ p}}\geq \displaystyle\frac{\mathcal{S}\|\nabla u_1\|_{p^{\ast}}^{ p/p^{\ast}}}{\|u_1\|_{ p}^{ p}}> \displaystyle\frac{\mathcal{S}}{|\Omega|^{1- p/p^{\ast}}}=\widetilde{\lambda},
\end{equation}
with $\widetilde{\lambda}$ as set in Theorem \ref{T2}. Moreover, we point out last inequality in \eqref{lan2} is strict since  $\mathcal{S}$ is not attained at $u_1$.

While, by part $(iii)$ of Lemma \ref{ein2} if $\lambda_{k+m}<\lambda_{k+m+1}$ then $i(\widetilde{\Psi}^{k+m})=k+m$. In order to apply Proposition \ref{pro1.0.23}, we need to construct a symmetric, compact subset $A_0$ of $\widetilde{\Psi}^{k+m}$ with the same index. For this, we first provide the following existence result.

\begin{lemma}\label{lem2.2.8}
For any $w\in L^p(\Omega)$, problem
\begin{equation}\label{kir2}
\left\{
\begin{array}{rcll}
(-\Delta)_p u+(-\Delta_p)^s u & = & |w|^{p-2}w  & \mbox{ in }\Omega,\\
u & = & 0 & \mbox{ on }  \mathbb{R}^{N} \setminus \Omega.
\end{array}
\right.
\end{equation}
admits a unique weak solution $u\in \mathrm{X}_0(\Omega)$. Furthermore, the operator $H:L^p(\Omega)\longrightarrow \mathrm{X}_0(\Omega)$ such that $w\mapsto H(w)=u$, with $u$ the solution of \eqref{kir2}, is continuous.
\end{lemma}
\begin{proof}
The existence of solution for \eqref{kir2} follows by a classical minimization argument. The uniqueness can be proved by property $(A_5)$ in Lemma \ref{Ap}.
\end{proof}

\begin{lemma}\label{lem2.2.9}
If $\lambda_k<\lambda_{k+1}$, then $\widetilde{\Psi}^{\lambda_k}$ admits a symmetric, compact subset $A_0$ with $i(A_0)=k$.
\end{lemma}
\begin{proof}
Let
$$\pi_{ p}(u)=\displaystyle\frac{u}{\|u\|_{ p}},\quad u\in \mathrm{X}_0(\Omega)\setminus\{0\}$$
be the radial projection into $\mathcal{N}_{ p}=\{u\in \mathrm{X}_0(\Omega) :\,\, \|u\|_p=1\}$, and set
$$A=\pi_{ p}(\widetilde{\Psi}^{\lambda_k})=\{w\in \mathcal{N}_{ p}:\,\, \|w\|_{\mathrm{X}_0}^p\leq \lambda_k\}.$$
Then, by part $(i)$ of Proposition \ref{pro1.0.21} and part $(iii)$ of Lemma \ref{ein2}, we get
$$i(A)=i(\widetilde{\Psi}^{\lambda_k})=k.$$
Let $u=H(v)$ for $v\in A$, with $H$ the operator defined in Lemma \ref{lem2.2.8}. We have
\begin{align*}
\|u\|_{\mathrm{X}_0}^{ p}&=\int_{\Omega} |v|^{ p-2}uv\,dx\leq \|v\|_p^{p-1 } \|u\|_p=\|u\|_p,
\\
1&=\|v\|_p^p=\int_{\Omega} |\nabla u|^{p-2}\nabla u\cdot\nabla v\,dx +\iint_{\mathbb{R}^{2N}}\mathcal{A}u(x,y)[v(x)-v(y)]\,dx\,dy\\
 & \leq \|u\|_{\mathrm{X}_0}^{ p-1}\|v\|_{\mathrm{X}_0}.
\end{align*}
which yields
$$\displaystyle\frac{\|u\|_{\mathrm{X}_0}^{ p}}{\|u\|_{ p}}\leq 1 \leq \|u\|_{\mathrm{X}_0}^{ p-1}\|v\|_{\mathrm{X}_0}
\quad\Longrightarrow\quad \|\pi_{ p}(u)\|_{\mathrm{X}_0}=\displaystyle\frac{\|u\|_{\mathrm{X}_0}}{\|u\|_{ p}}\leq \|v\|_{\mathrm{X}_0},$$
which means $\pi_{ p}(H(A))\subseteq A$. Set $\overline{H}=\pi_{ p}\circ H$ and let $\overline{A}=\overline{H}(A)$.
By compactness of embedding $\mathrm{X}_0(\Omega)\hookrightarrow L^p(\Omega)$ and the continuity of operator  $H$ in Lemma \ref{lem2.2.8}, we get  $\overline{A}$ is compact in $\mathrm{X}_0(\Omega)$. On the other hand, since $\overline{A}\subset A$ and $\overline{H}$ is an odd continuous function from $A$ to $\overline{A}$, by part $(i)$ of Proposition \ref{pro1.0.21} we have  $i(A)=i(\overline{A})=k$.
Finally, let
$$\pi(u)=\displaystyle\frac{u}{\|u\|_{\mathrm{X}_0}}, \quad u\in  \mathrm{X}_0(\Omega)\setminus\{0\}$$
be the radial projection into $\mathcal{N}$. Set $A_0=\pi(\overline{A})$, then $A_0\subseteq \widetilde{\Psi}^{\lambda_k}$ is compact in $\mathrm{X}_0(\Omega)$ and satisfies $i(\overline{A})=i(A_0)=k$.
\end{proof}

Finally, we can address the proof of Theorem \ref{T2}.

\begin{proof}[Proof of Theorem \ref{T2}]
We only give the proof of part $(ii)$, since the proof of $(i)$ is similar and simpler.

Let $k$, $m\in \mathbb{N}$ be fixed. Let
$$\lambda_k\leq \lambda<\lambda_{k+1}=\ldots=\lambda_{k+m}<\lambda_{k+m+1}\quad\mbox{and}\quad \lambda>\lambda_{k+1}-\widetilde{\lambda},$$
with $\widetilde{\lambda}=\frac{\mathcal{S}}{|\Omega|^{p/N}}$.
By Lemma \ref{lemmaPS} functional $J_\lambda$ satisfies the $(PS)_c$ condition for
$$c<\frac{1}{N}\mathcal{S}^{N/p}.$$
Hence, the idea is to apply Proposition \ref{pro1.0.23} with $b=\frac{1}{N}\mathcal{S}^{N/p}$.

By Lemma \ref{lem2.2.9}, $\widetilde{\Psi}^{\lambda_{k+m}}$ admits a symmetric, compact subset $A_0$ with $i(A_0)=k+m$. We denote $B_0=\widetilde{\Psi}_{\lambda_{k+1}}$ so that by part $(iii)$ of Lemma \ref{ein2} we have $i(\mathcal{N}\,\setminus B_0)=k$.
Let $0<r<R$ and let $A$, $B$ and $X$ be as in Proposition \ref{pro1.0.23}. For any $u\in B_0$, by \eqref{sob1}
\begin{equation}
\begin{aligned}\label{inff1}
J_\lambda(ru)&=\frac{r^{ p}}{ p}\|u\|_{\mathrm{X}_0}^{ p}-\frac{\lambda r^{ p}}{ p}\|u\|_p^p -\frac{r^{p^{\ast}}}{p^{\ast}}\|u\|_{p^{\ast}}^{p^{\ast}}
\geq \frac{r^{ p}}{ p}\|u\|_{\mathrm{X}_0}^{ p}-\frac{\lambda r^{ p}}{ p}\,\frac{\|u\|_{\mathrm{X}_0}^{ p}}{\lambda_{k+1}} -\frac{r^{p^{\ast}}}{p^{\ast}}\|u\|_{\mathrm{X}_0}^{p^{\ast}} \mathcal{S}^{-p^{\ast}/p}
\\[0.35cm]
&= \frac{r^{ p}}{ p}\|u\|_{\mathrm{X}_0}^{ p}\left(1-\frac{\lambda}{\lambda_{k+1}}\right)-\frac{r^{p^{\ast}}}{p^{\ast}}\|u\|_{\mathrm{X}_0}^{p^{\ast}} \mathcal{S}^{-p^{\ast}/p}.
\end{aligned}
\end{equation}
Since $\lambda<\lambda_{k+1}$, it follows from  \eqref{inff1} that $\displaystyle \inf_{u\in B} J_\lambda(u)>0$ for $r$ sufficiently small. For $u\in A_0\subset\widetilde{\Psi}^{\lambda_{k+m}}$, by H\"older's inequality
\begin{equation}
\begin{aligned}\label{supp1}
 J_\lambda(Ru)&=\frac{R^{ p}}{ p}\|u\|_{\mathrm{X}_0}^{ p}-\frac{\lambda R^{ p}}{ p}\|u\|_p^p -\frac{R^{p^{\ast}}}{p^{\ast}}\|u\|_{p^{\ast}}^{p^{\ast}}
\leq \frac{R^{ p}}{ p}\|u\|_{\mathrm{X}_0}^{ p}-\frac{\lambda R^{ p}}{ p}\,\frac{\|u\|_{\mathrm{X}_0}^{ p}}{\lambda_{k+m}}-\frac{R^{p^{\ast}}}{p^{\ast}}\left(\displaystyle
\frac{\|u\|_p^p}{|\Omega|^{1-\frac{ p}{p^{{\ast}}}}}\right)^{p^{\ast}/ p}
 \\[0.4cm]
 &\leq \frac{R^{ p}}{ p}\|u\|_{\mathrm{X}_0}^{ p}-\frac{\lambda R^{ p}}{ p}\,\frac{\|u\|_{\mathrm{X}_0}^{ p}}{\lambda_{k+m}}-\frac{R^{p^{\ast}}}{p^{\ast}}\left(\displaystyle\frac{ \|u\|_{\mathrm{X}_0}^{p}}{\lambda_{k+m}\,|\Omega|^{1-\frac{p}{p^{{\ast}}}}}\right)^{p^{\ast}/ p}.
\end{aligned}
\end{equation}
Since $\lambda<\lambda_{k+m}$, by \eqref{supp1} we take $R>r$ sufficiently large so that
$J_\lambda(u)\leq 0$, for any $u\in A$. For $u\in X$, by the H\"older's inequality
\begin{equation}
\begin{aligned}\label{neu1}
 J_\lambda(u)&=\frac{1}{ p}\|u\|_{\mathrm{X}_0}^{ p}-\frac{\lambda }{ p}\|u\|_p^p -\frac{1}{p^{\ast}}\|u\|_{p^{\ast}}^{p^{\ast}}
\leq \frac{\lambda_{k+m}-\lambda}{ p}\|u\|_p^p -\frac{1}{p^{\ast}}\left(\displaystyle\frac{\|u\|_p^p}{|\Omega|^{1-\frac{ p}{p^{\ast}}}}\right)^{p^{\ast}/ p}
\\[0.4cm]
 &=\frac{\lambda_{k+m}-\lambda}{p}\|u\|_p^p -\frac{1}{p^{\ast}}\left(\frac{\|u\|_p^{ p}}{|\Omega|^{1-\frac{ p}{p^{\ast}}}}\right)^{p^{\ast}/ p}.
\end{aligned}
\end{equation}
Let us set $g_\lambda:\mathbb{R}_0^{+}\longrightarrow \mathbb{R}$ given by
$$ g_\lambda(t)= \left(\displaystyle\frac{\lambda_{k+m}-\lambda}{ p}\right)t -\frac{1}{p^{\ast}}\left(\displaystyle\frac{t}{|\Omega|^{1-\frac{ p}{p^{\ast}}}}\right)^{p^{\ast}/ p}.$$
Since $\lambda_{k+m}>\lambda$, function $g_\lambda$ attains its maximum at
 $t_0=(\lambda_{k+m}-\lambda)^{{ p}/{(p^{\ast}- p)}}\,|\Omega|$, with
$$g_\lambda(t_0)=\left(\frac{1}{ p}-\frac{1}{p^{\ast}}\right)(\lambda_{k+m}-\lambda)^{\frac{p^{\ast}}{p^{\ast}- p}}|\Omega|.$$
Thus,
$$
\displaystyle  \sup_{u\in X} J_\lambda(u)\leq \left(\frac{1}{ p}-\frac{1}{p^{\ast}}\right)(\lambda_{k+m}-\lambda)^{\frac{p^{\ast}}{p^{\ast}- p}}|\Omega|\leq \left(\frac{1}{ p}-\frac{1}{p^{\ast}}\right)\,\mathcal{S}^{\frac{ p^{\ast}}{p^{\ast}- p}}=\frac{1}{N}\mathcal{S}^{N/p}
$$
since $\lambda_{k+m}-\widetilde{\lambda}<\lambda$. By Proposition \ref{pro1.0.23} we get $m$ distinct pairs of nontrivial critical points  $\pm u_{j,\lambda}$, with $j=1,\ldots,m$, for $J_\lambda$, that is
$$0<J_\lambda(u_{j,\lambda})\leq \left(\frac{1}{ p}-\frac{1}{p^{\ast}}\right)(\lambda_{k+m}-\lambda)^{\frac{p^{\ast}}{p^{\ast}- p}}|\Omega|\to 0,\quad\mbox{as }\quad \lambda\to\lambda_{k+m},$$
which implies that
$$\left(\frac{1}{ p}-\frac{1}{p^{\ast}}\right) \|u_{j,\lambda}\|_{p^{\ast}}^{p^{\ast}}=J_\lambda(u_{j,\lambda})
-\frac{1}{ p}\langle J^{\prime}_\lambda(u_{j,\lambda}),u_{j,\lambda}\rangle=J_\lambda(u_{j,\lambda})\to 0,$$
as $\lambda\to \lambda_{k+m}$. Hence, the H\"older's inequality implies that $u_{j,\lambda}\to 0$ in $L^p(\Omega)$, as $\lambda\to \lambda_{k+m}$, and so it follows
$$\|u_{j,\lambda}\|_{\mathrm{X}_0}^{ p}=\left( p I_{\lambda}(u_{j,\lambda})+\lambda \|u_{j,\lambda}\|_{ p}^{ p}+\frac{ p}{p^{\ast}}\|u_{j,\lambda}\|_{p^{\ast}}^{p^{\ast}}\right)\to 0,$$
as $\lambda\to\lambda_{k+m}$, concluding the proof.

\end{proof}

%%%%%%%%%%%%%%%%%%%%%%%%%%%%%%%%%%%%%%%%%%%%%%%%%%%%%%%%%%%%%%%%%%%%%%%%%%%%%%%%%%%%%%%%%%%%%%%%%%%%%%%%%%%%%%%%%%%%%%%%%%%

\section{Superlinear case: a mountain pass solution}\label{sec5}

In this section, we prove Theorem \ref{supex}. For this, we assume along the section that $q\in(p,p^{\ast})$ without further mentioning. We apply the classical mountain pass theorem. We first need a suitable control from above for functional $J_{\lambda}$.

The idea is to employ a suitable truncation of the function
\begin{equation}\label{ueps}
\mathrm{U}_\varepsilon(x)=\frac{K_{N,p}\,\varepsilon^{(N-p)/p(p-1)}}{(\varepsilon^{p/(p-1)}+|x|^{p/(p-1)})^{(N-p)/p}},\quad\mbox{with }\varepsilon>0,
\end{equation}
which is in $W^{1,p}(\mathbb R^N)$, where the best constant Sobolev inclusion is attained, considering normalization constant $K_{N,p}>0$ given by
$$
K_{N,p}=\left[N\left(\frac{N-p}{p-1}\right)^{p-1}\right]^{(N-p)/p^2}.
$$
Let us fix $r>0$ such that $\overline{B_{4r}(0)}\subset\Omega$ and let us introduce a radial cut-off function $\phi_r\in C^\infty(\mathbb R^N,[0,1])$ such that
\begin{equation}\label{phi}
\phi_r(x)=\begin{cases}
1 & \mbox{ if }x\in B_{r}(0),\\
0 & \mbox{ if }x\in B_{2r}^c(0).
\end{cases}
\end{equation}
For any $\varepsilon>0$, we set
\begin{equation}\label{veps}
u_\varepsilon=\phi_r \,\mathrm{U}_\varepsilon\quad\mbox{and}\quad
v_\varepsilon=\frac{u_\varepsilon}{\|u_\varepsilon\|_{p^{\ast}}},
\end{equation}
which of course are radial and belong to $W_0^{1,p}(\Omega)$.
Arguing as in the proof of \cite[Theorem 8.4]{GA1}, see also \cite{DH,GA2,GV} and above all \cite{AD,SY}, we have the following estimates.

\begin{lemma}\label{lemmanorma}
Let $v_\varepsilon$ be as in \eqref{veps}. Then,
$$
\|\nabla v_\varepsilon\|_p^p=\mathcal{S}+\mathrm{O}(\varepsilon^{(N-p)/(p-1)})
$$
as $\varepsilon\to0^+$.
\end{lemma}

\begin{lemma}\label{lemmanormaq}
Let $v_\varepsilon$ be as in \eqref{veps}. Then, there exists $C>0$ such that
$$
\|v_\varepsilon\|_q^q\geq C\varepsilon^{N-q(N-p)/p}
$$
for any $q>p^{\ast}(1-1/p)$.
\end{lemma}

\begin{remark}\label{osservazione}
We strongly point out that we can apply Lemma \ref{lemmanormaq} when $m_{N,p,s}>\beta_{N,q,p}$. Indeed, by \eqref{parametri} we get
$$
m_{N,p,s}>\beta_{N,q,p}\quad\Longrightarrow\quad q>p^{\ast}-\frac{p}{N-p}m_{N,p,s}\geq p^{\ast}-\frac{p}{p-1}.
$$
Hence, considering also that $p<q<p^*$, in this case we have
$$q>\max\left\{p,p^{\ast}-\frac{p}{p-1}\right\}\geq p^{\ast}\left(1-\frac{1}{p}\right)$$
with the last inequality well discussed in \cite[Remark 3.4]{GA2}.
\end{remark}

Concerning the Gagliardo seminorm, inspired by \cite[Proposition 21]{SV} we claim the following estimate.

\begin{lemma}\label{lemmagaglia}
Let $u_\varepsilon$ be as in \eqref{veps}. Then,
$$
\displaystyle\iint_{\mathbb R^{2N}}\frac{|u_\varepsilon(x)-u_\varepsilon(y)|^p}{|x-y|^{N+ps}}\,dx\,dy= \mathrm{O}(\varepsilon^{m_{N,p,s}})
$$
as $\varepsilon\to0^+$, with $m_{N,p,s}$ given in \eqref{parametri}.
\end{lemma}
\begin{proof}
Along the proof, we denote with $C>0$ possibly different constants, but independent by $\varepsilon>0$.

Let $\varrho>0$, by \eqref{ueps} and \eqref{veps}, for $x\in B_\varrho^c(0)$ we have
\begin{equation}\label{varrho1}
|u_\varepsilon(x)|\leq\frac{K_{N,p}\,\varepsilon^{(N-p)/p(p-1)}}{(\varepsilon^{p/(p-1)}+\varrho^{p/(p-1)})^{(N-p)/p}}\leq C\,\varepsilon^{(N-p)/p(p-1)},
\end{equation}
while
\begin{equation}\label{varrho2}
\begin{aligned}
|\nabla u_\varepsilon(x)|=&\left|\nabla\phi_r(x)\frac{K_{N,p}\,\varepsilon^{(N-p)/p(p-1)}}{(\varepsilon^{p/(p-1)}+|x|^{p/(p-1)})^{(N-p)/p}}\right.\\
&\left.+\phi_r(x)\frac{N-p}{p-1}|x|^{(2-p)/(p-1)}x\,\frac{K_{N,p}\,\varepsilon^{(N-p)/p(p-1)}}{(\varepsilon^{p/(p-1)}+|x|^{p/(p-1)})^{N/p}}\right|\\
\leq& C\varepsilon^{(N-p)/p(p-1)}\left[\frac{1}{(\varepsilon^{p/(p-1)}+|x|^{p/(p-1)})^{(N-p)/p}}
+\frac{|x|^{p/(p-1)}}{\varrho}\frac{1}{(\varepsilon^{p/(p-1)}+|x|^{p/(p-1)})^{N/p}}\right]\\
\leq& C\varepsilon^{(N-p)/p(p-1)}
\end{aligned}
\end{equation}
since $|x|\geq\varrho$.

If $x\in\mathbb R^N$ and $y\in B_r^c(0)$ with $|x-y|\leq r/2$, by a first order Taylor expansion it follows that
$$
|u_\varepsilon(x)-u_\varepsilon(y)|\leq|\nabla u(\xi)||x-y|
$$
with $\xi$ belonging to the segment joining $x$ and $y$, namely $\xi=tx+(1-t)y$ for some $t\in[0,1]$.
From this
$$
|\xi|=|y+t(x-y)|\geq|y|-t|x-y|\geq r-tr/2\geq r/2,
$$
so that by \eqref{varrho2} with $\varrho=r/2$, we get
\begin{equation}\label{claim10a}
|u_\varepsilon(x)-u_\varepsilon(y)|\leq C\varepsilon^{(N-p)/p(p-1)}|x-y|\qquad\mbox{ if }\,x\in\mathbb R^N,\, y\in B_r^c(0)\,\mbox{ with }\,|x-y|\leq r/2.
\end{equation}
While, by combining \eqref{claim10a} and \eqref{varrho1} with $\varrho=r$, we have
\begin{equation}\label{claim10b}
|u_\varepsilon(x)-u_\varepsilon(y)|\leq C\varepsilon^{(N-p)/p(p-1)}\min\left\{1,|x-y|\right\}\qquad\mbox{ if }\,x, y\in B_r^c(0).
\end{equation}
Now, let us set
$$
\begin{aligned}
\mathbb D:=\left\{(x,y)\in\mathbb R^{2N}:\,\,x\in B_r(0),\,\,y\in B_r^c(0)\,\mbox{ and }\,|x-y|>r/2\right\},\\
\mathbb E:=\left\{(x,y)\in\mathbb R^{2N}:\,\,x\in B_r(0),\,\,y\in B_r^c(0)\,\mbox{ and }\,|x-y|\leq r/2\right\},
\end{aligned}
$$
in this way, we can split
\begin{equation}\label{4.23}
\begin{aligned}
\displaystyle\iint_{\mathbb R^{2N}}\frac{|u_\varepsilon(x)-u_\varepsilon(y)|^p}{|x-y|^{N+ps}}\,dx\,dy=&
\displaystyle\iint_{B_r(0)\times B_r(0)}\frac{|\mathrm{U}_\varepsilon(x)-\mathrm{U}_\varepsilon(y)|^p}{|x-y|^{N+ps}}\,dx\,dy\\
&+\displaystyle2\iint_{\mathbb D}\frac{|u_\varepsilon(x)-u_\varepsilon(y)|^p}{|x-y|^{N+ps}}\,dx\,dy\\
&+\displaystyle2\iint_{\mathbb E}\frac{|u_\varepsilon(x)-u_\varepsilon(y)|^p}{|x-y|^{N+ps}}\,dx\,dy\\
&+\displaystyle\iint_{B_r^c(0)\times B_r^c(0)}\frac{|u_\varepsilon(x)-u_\varepsilon(y)|^p}{|x-y|^{N+ps}}\,dx\,dy.
\end{aligned}
\end{equation}
By \eqref{claim10b} we have
\begin{equation}\label{4.24}
\begin{aligned}
\displaystyle\iint_{B_r^c(0)\times B_r^c(0)}\frac{|u_\varepsilon(x)-u_\varepsilon(y)|^p}{|x-y|^{N+ps}}\,dx\,dy
&\leq C \varepsilon^{(N-p)/(p-1)}\displaystyle\iint_{B_{2r}(0)\times \mathbb R^N}\frac{\min\left\{1,|x-y|^p\right\}}{|x-y|^{N+ps}}\,dx\,dy\\
&=  \mathrm{O}(\varepsilon^{(N-p)/(p-1)})
\end{aligned}
\end{equation}
while by \eqref{claim10a} we get
\begin{equation}\label{4.25}
\begin{aligned}
\displaystyle\iint_{\mathbb E}\frac{|u_\varepsilon(x)-u_\varepsilon(y)|^p}{|x-y|^{N+ps}}\,dx\,dy
&\leq C \varepsilon^{(N-p)/(p-1)}\displaystyle\iint_{\substack{x\in B_{r}(0),\,y\in B_{r}^c(0)\\|x-y|\leq r/2}}\frac{|x-y|^p}{|x-y|^{N+ps}}\,dx\,dy\\
&\leq C \varepsilon^{(N-p)/(p-1)}\displaystyle\int_{B_{r}(0)}dx\int_{B_{r/2}(0)}\frac{1}{|\xi|^{N+ps-p}}d\xi
= \mathrm{O}(\varepsilon^{(N-p)/(p-1)}).
\end{aligned}
\end{equation}
Since by \eqref{phi} we have $u_\varepsilon(x)=\mathrm{U}_\varepsilon(x)$ for any $x\in B_r(0)$, then for any $(x,y)\in\mathbb D$ by convexity
$$
|u_\varepsilon(x)-u_\varepsilon(y)|^p=|\mathrm{U}_\varepsilon(x)-u_\varepsilon(y)|^p\leq 2^{p-1}\left(|\mathrm{U}_\varepsilon(x)-\mathrm{U}_\varepsilon(y)|^p+|\mathrm{U}_\varepsilon(y)-u_\varepsilon(y)|^p\right),
$$
from which
\begin{equation}\label{4.27}
\begin{aligned}
\iint_{\mathbb D}\frac{|u_\varepsilon(x)-u_\varepsilon(y)|^p}{|x-y|^{N+ps}}\,dx\,dy\leq
2^{p-1}\left(\displaystyle\iint_{\mathbb D}\frac{|\mathrm{U}_\varepsilon(x)-\mathrm{U}_\varepsilon(y)|^p}{|x-y|^{N+ps}}\,dx\,dy
+\displaystyle\iint_{\mathbb D}\frac{|\mathrm{U}_\varepsilon(y)-u_\varepsilon(y)|^p}{|x-y|^{N+ps}}\,dx\,dy\right).
\end{aligned}
\end{equation}
By \eqref{varrho1} with $\delta=r$
\begin{equation}\label{4.28}
\begin{aligned}
\displaystyle\iint_{\mathbb D}\frac{|\mathrm{U}_\varepsilon(y)-u_\varepsilon(y)|^p}{|x-y|^{N+ps}}\,dx\,dy
&\leq 2^{p-1}\displaystyle\iint_{\mathbb D}\frac{|\mathrm{U}_\varepsilon(y)|^p+|u_\varepsilon(y)|^p}{|x-y|^{N+ps}}\,dx\,dy\\
&\leq 2^{p}\displaystyle\iint_{\mathbb D}\frac{|\mathrm{U}_\varepsilon(y)|^p}{|x-y|^{N+ps}}\,dx\,dy\\
&\leq C\varepsilon^{(N-p)/(p-1)}\displaystyle\iint_{\substack{x\in B_{r}(0),\,y\in B_{r}^c(0)\\|x-y|> r/2}}\frac{1}{|x-y|^{N+ps}}\,dx\,dy\\
&\leq C\varepsilon^{(N-p)/(p-1)}\displaystyle\int_{x\in B_{r}(0)}dx\int_{B_{r/2}^c(0)}\frac{1}{|\xi|^{N+ps}}d\xi=  \mathrm{O}(\varepsilon^{(N-p)/(p-1)}).
\end{aligned}
\end{equation}
Finally, by a double change of variables $x=\varepsilon\xi$ and $y=\varepsilon\zeta$, we get
\begin{equation}\label{extra}
\begin{aligned}
\displaystyle\iint_{\mathbb R^{2N}}\frac{|\mathrm{U}_\varepsilon(x)-\mathrm{U}_\varepsilon(y)|^p}{|x-y|^{N+ps}}\,dx\,dy
=\varepsilon^{p-ps}\displaystyle\iint_{\mathbb R^{2N}}\frac{|U_1(\xi)-U_1(\zeta)|^p}{|\xi-\zeta|^{N+ps}}\,d\xi\,d\zeta= \mathrm{O}(\varepsilon^{p-ps}).
\end{aligned}
\end{equation}
Thus, since trivially
$$
\displaystyle\iint_{V}\frac{|\mathrm{U}_\varepsilon(x)-\mathrm{U}_\varepsilon(y)|^p}{|x-y|^{N+ps}}\,dx\,dy\leq
\displaystyle\iint_{\mathbb R^{2N}}\frac{|\mathrm{U}_\varepsilon(x)-\mathrm{U}_\varepsilon(y)|^p}{|x-y|^{N+ps}}\,dx\,dy
$$
for $V=B_r(0)\times B_r(0)$ and $V=\mathbb D$, then combining \eqref{4.23}-\eqref{extra} we can conclude the proof.
\end{proof}

Combining Lemmas \ref{lemmanorma}-\ref{lemmagaglia} and considering Remark \ref{osservazione}, we get the following control for $J_{\lambda}$.

\begin{lemma}\label{lemmastima}
Let $v_\varepsilon$ be as in \eqref{veps} and let $m_{N,p,s}$, $\beta_{N,q,p}$ be as in \eqref{parametri}. Then, we have the following cases:
\begin{itemize}
\item[$(i)$] if $m_{N,p,s}>\beta_{N,q,p}$, then there exists $\varepsilon> 0$ sufficiently small such that
$$
\sup_{t\geq0} J_{\lambda}(tv_\varepsilon)<\frac{1}{N}\mathcal{S}^{N/p}\quad\mbox{ for any }\lambda>0;
$$
\item[$(ii)$] if $m_{N,p,s}\leq \beta_{N,q,p}$, then for any $\varepsilon> 0$ there exists $\lambda^{\ast}=\lambda^{\ast}(\varepsilon)>0$ such that
$$
\sup_{t\geq0} J_{\lambda}(tv_\varepsilon)<\frac{1}{N}\mathcal{S}^{N/p}\quad\mbox{ for any }\lambda\geq\lambda^{\ast}.
$$
\end{itemize}
\end{lemma}
\begin{proof}
Let $\varepsilon>0$ and $\lambda>0$, we have
\begin{equation}\label{tveps}
J_\lambda(tv_\varepsilon)=\frac{t^p}{p}\left(\|\nabla v_\varepsilon\|_p^p+[v_\varepsilon]_{s,p}^p\right)-\lambda\frac{t^q}{q}\|v_\varepsilon\|_q^q-\frac{t^{p^{\ast}}}{p^{\ast}},
\end{equation}
from which we distinguish the two cases.

\vspace{0.3cm}
$\bullet$\, {\em Case (i).}
Since by \cite[Lemma 7.1]{GA1} we know that
$$
\|u_\varepsilon\|_{p^{\ast}}^{p^{\ast}}=\mathcal{S}^{N/p}+ \mathrm{O}(\varepsilon^{N/(p-1)}),
$$
by Lemmas \ref{lemmanorma}-\ref{lemmagaglia} it follows that
\begin{equation}\label{stima}
J_\lambda(tv_\varepsilon)\leq\frac{t^p}{p}\left(\mathcal{S}+  \mathrm{O}(\varepsilon^{m_{N,p,s}})\right)-C\lambda\frac{t^q}{q}\varepsilon^{\beta_{N,q,p}}-\frac{t^{p^{\ast}}}{p^{\ast}}.
\end{equation}
Let us set
$$
f_{\varepsilon,\lambda}(t):=\frac{t^p}{p}\left(\mathcal{S}+C\varepsilon^{m_{N,p,s}}\right)-C\lambda\frac{t^q}{q}\varepsilon^{\beta_{N,q,p}}-\frac{t^{p^{\ast}}}{p^{\ast}},
$$
such that $f_{\varepsilon,\lambda}(0)=0$ and $f_{\varepsilon,\lambda}(t)\to-\infty$ as $t\to\infty$. Hence, there exists $t_{\varepsilon,\lambda}\geq0$ such that
$$
\sup_{t\geq0}f_{\varepsilon,\lambda}(t)=f_{\varepsilon,\lambda}(t_{\varepsilon,\lambda}).
$$
If $t_{\varepsilon,\lambda}=0$, then we can conclude easily the proof of the lemma.
While, for $t_{\varepsilon,\lambda}>0$, we get
\begin{equation}\label{gprimo}
0=f^{\prime}_{\varepsilon,\lambda}(t_{\varepsilon,\lambda})=t_{\varepsilon,\lambda}^{p-1}\left(\mathcal{S}+C\varepsilon^{m_{N,p,s}}\right)
-C\lambda t_{\varepsilon,\lambda}^{q-1}\varepsilon^{\beta_{N,q,p}}-t_{\varepsilon,\lambda}^{p^{\ast}-1},
\end{equation}
from which
$$
t_{\varepsilon,\lambda}<\left(\mathcal{S}+C\varepsilon^{m_{N,p,s}}\right)^{1/(p^{\ast}-p)}.
$$
On the other hand, for $\varepsilon>0$ sufficiently small, by \eqref{gprimo} we get $t_{\varepsilon,\lambda}\geq\mu_\lambda>0$ so that, using also the fact that the map
$$
t\longmapsto \frac{t^p}{p}\left(\mathcal{S}+C\varepsilon^{m_{N,p,s}}\right)-\frac{t^{p^{\ast}}}{p^{\ast}}
$$
is increasing in the closed interval $\left[0,\left(\mathcal{S}+C\varepsilon^{m_{N,p,s}}\right)^{1/(p^{\ast}-p)}\right]$ containing $t_{\varepsilon,\lambda}$, we obtain
$$
\begin{aligned}
\sup_{t\geq0} f_{\varepsilon,\lambda}(t)&=f_{\varepsilon,\lambda}(t_{\varepsilon,\lambda})\\
&<\frac{\left(\mathcal{S}+C\varepsilon^{m_{N,p,s}}\right)^{1+p/(p^{\ast}-p)}}{p}-C\varepsilon^{\beta_{N,q,p}}-\frac{\left(\mathcal{S}+C\varepsilon^{m_{N,p,s}}\right)^{p^{\ast}/(p^{\ast}-p)}}{p^{\ast}}\\
&=\frac{1}{N}\left(\mathcal{S}+C\varepsilon^{m_{N,p,s}}\right)^{N/p}-C\varepsilon^{\beta_{N,q,p}}\\
&\leq\frac{1}{N}\mathcal{S}^{N/p}+C\varepsilon^{m_{N,p,s}}-C\varepsilon^{\beta_{N,q,p}}\\
&<\frac{1}{N}\mathcal{S}^{N/p}
\end{aligned}
$$
for still $\varepsilon>0$ sufficiently small, since $m_{N,p,s}>\beta_{N,q,p}$ in the last inequality. This concludes the proof joint with \eqref{stima}.

\vspace{0.3cm}
$\bullet$\, {\em Case (ii).}
We can not apply the previous proof, as discussed in Remark \ref{osservazione}. Since $p<q<p^{\ast}$, by \eqref{tveps} we have $J_\lambda(tv_\varepsilon)\to-\infty$ as $t\to\infty$. From this, considering also that $J_\lambda(0)=0$ and $J_\lambda$ is continuous, there exists $\overline{t}_{\lambda,\varepsilon}>0$ such that $J_\lambda(\overline{t}_{\lambda,\varepsilon}v_\varepsilon)=\sup_{t\geq0} J_\lambda(tv_\varepsilon)$.
Hence, we have
$$\langle J^{\prime}_\lambda(\overline{t}_{\lambda,\varepsilon}v_\varepsilon),v_\varepsilon\rangle=0$$
which yields
\begin{equation}\label{stima1}
\overline{t}_{\lambda,\varepsilon}^{\,p-1}\|v_\varepsilon\|_{\mathrm{X}_0}^p
-\lambda \overline{t}_{\lambda,\varepsilon}^{\,q-1}\|v_\varepsilon\|_q^q-\overline{t}_{\lambda,\varepsilon}^{\,p^{\ast}-1}
=0.
\end{equation}
From this, we get that $\{\overline{t}_{\lambda,\varepsilon}\}_\lambda\subset\mathbb R^+$ is bounded since $p<q<p^{\ast}$. Thus, combining \eqref{stima1} with
$$
\lambda \overline{t}_{\lambda,\varepsilon}^{\,q-1}\|v_\varepsilon\|_q^q\to\infty,\quad\mbox{as }\lambda\to\infty,
$$
we deduce that $\overline{t}_{\lambda,\varepsilon}\to0$ as $\lambda\to\infty$. From this, since $J_\lambda$ is continuous, we obtain
$$
\sup_{t\geq0} J_\lambda(tv_\varepsilon)=J_\lambda(\overline{t}_{\lambda,\varepsilon}v_\varepsilon)\to0,\quad\mbox{as }\lambda\to\infty,
$$
concluding the proof.
\end{proof}

Now, we can verify the mountain pass geometry for $J_\lambda$.

\begin{lemma}\label{lemmageometry}
Let $\lambda>0$ and $\varepsilon>0$ be set as in Lemma \ref{lemmastima}. Then, there exist $\alpha$, $\rho>0$ such that:
\begin{itemize}
\item [$(i)$] for any $u\in \mathrm{X}_0(\Omega)$ with $\|u\|_{\mathrm{X}_0}=\rho$ we have $J_\lambda(u)\geq\alpha$;
\item [$(ii)$] there exists a $t_\varepsilon>0$ sufficiently large, so that $\|t_\varepsilon v_\varepsilon\|_{\mathrm{X}_0}>\rho$ and $J_\lambda(t_\varepsilon v_\varepsilon)<\alpha$, with $v_\varepsilon$ as given in Lemma \ref{lemmastima}.
\end{itemize}
\end{lemma}
\begin{proof}
By the Sobolev embedding and \eqref{sob1}, we have
$$
J_\lambda(u)\geq\frac{1}{p}\|u\|_{\mathrm{X}_0}^p-\lambda C_q\|u\|_{\mathrm{X}_0}^q-\mathcal{S}^{p^{\ast}/p}\|u\|_{\mathrm{X}_0}^{p^{\ast}},
$$
for any $u\in \mathrm{X}_0(\Omega)$. Since $p<q<p^{\ast}$, we can easily get part $(i)$ assuming $\|u\|_{\mathrm{X}_0}$ sufficiently small.
On the other hand, we have
$$
\lim_{t\to\infty}J_\lambda(tv_\varepsilon)=-\infty,
$$
from which we can conclude the proof.
\end{proof}

Now, we are able to prove Theorem \ref{supex}.

\begin{proof}[\bf Proof of Theorem \ref{supex}]
Let $\lambda>0$ and $\varepsilon>0$ be set as in Lemma \ref{lemmastima}.
Thanks to Lemma \ref{lemmageometry}, we can set the critical mountain pass level as
\begin{equation}\label{eq3.6}
m_\lambda=\inf_{\xi\in\Gamma}\max_{t\in[0,1]}J_\lambda(\xi(t)),
\end{equation}
where
$$
\Gamma=\left\{\xi\in C^0([0,1];\mathrm{X}_0(\Omega)):\,\,\xi(0)=0,\,\, \xi(1)=t_\varepsilon v_\varepsilon\right\}.
$$
By Lemma \ref{lemmastima} we have
$$
m_\lambda\leq\sup_{t\geq0} J_{\lambda}(tv_\varepsilon)<\frac{1}{N}\mathcal{S}^{N/p},
$$
in both cases $(i)$ and $(ii)$, so that we can apply part $(i)$ of Lemma \ref{lemmaPS} at level $m_\lambda$. Thus, condition $(PS)_{m_\lambda}$ is verified and by the classical mountain pass theorem, see \cite[Theorem 1.15]{W}, we can conclude the proof.

\end{proof}

%%%%%%%%%%%%%%%%%%%%%%%%%%%%%%%%%%%%%%%%%%%%%%%%%%%%%%%%%%%%%%%%%%%%%%%%%%%%%%%%%%%%%%%%%%%%%%%%%%%%%%%%%%%%%%%%%%%%%%%%%%%

\section{Superlinear case: category theory}\label{sec6}

In this section, we prove Theorem \ref{T1}. For this, we assume along the section that $2\leq p<q<p^{\ast}$ satisfy $m_{N,p,s}>\beta_{N,q,p}$ given in \eqref{parametri}, without further mentioning.

In order to get the existence of multiple critical points for the functional $J_\lambda$, we will use the Lusternik-Schnirelman theory, as developed  in \cite{LS}. This approach allows to find critical points of our functional $J_\lambda$ on a suitable manifold, exploiting the topological properties of the manifold itself. For this, we set the \textit{Nehari manifold} associated to $J_\lambda$ as
$$
\mathcal N_\lambda=\left\{v\in \mathrm{X}_0(\Omega)\setminus\{0\}:\,\,\langle J^{\prime}_\lambda(v),v\rangle=0\right\}.
$$
It is easy to see that $\mathcal N_\lambda$ is not empty.

The idea is to apply the theory as written in \cite[Chapter 5]{W}. For this, let us define $\psi_\lambda:\mathrm{X}_0(\Omega)\to\mathbb R$ as
$$
\psi_\lambda(u)=\langle J^{\prime}_\lambda(u),u\rangle=\|\nabla u\|_p^p+[u]_{s,p}^p-\lambda\|u\|_q^q-\|u\|_{p^{\ast}}^{p^{\ast}}.
$$
We have that $\psi_\lambda\in C^2(\mathrm{X}_0(\Omega);\mathbb R)$ being $p\geq2$, with in particular
$$
\langle \psi^{\prime}_\lambda(u),u\rangle=p\left(\|\nabla u\|_p^p+[u]_{s,p}^p\right)-\lambda q\|u\|_q^q-p^{\ast}\|u\|_{p^{\ast}}^{p^{\ast}}.
$$
From this, we obtain
\begin{align}\label{2.11}
\langle \psi^{\prime}_\lambda(u),u\rangle=\lambda(p-q)\|u\|_q^q+(p-p^{\ast})\|u\|_{p^{\ast}}^{p^{\ast}}<0\qquad\mbox{for any }u\in\mathcal N_\lambda,
\end{align}
so that the structural assumptions on \cite[Section 5.3]{W} are satisfied by $J_\lambda$, $\psi_\lambda$ and $\mathcal N_\lambda$.

In the spirit of \cite{W}, we look for critical points of $J_\lambda$ on $\mathcal N_\lambda$. In order to construct these critical points, by \cite[Proposition 5.12 and Definition 5.17]{W} we set a new compactness condition. We say that functional $J_\lambda$ restricted on $\mathcal N_\lambda$ satisfies the Palais-Smale condition at level $c$, $(PS)_c^{\mathcal N_\lambda}$ for short, if any sequence $\{u_n\}_n\subset\mathcal N_\lambda$ such that
\begin{align}\label{palais-smale-restrict}
	J_\lambda(u_n)\to c \quad \text{and}\quad
	\min_{\gamma\in\mathbb R}\|J^{\prime}_\lambda(u_n)-\gamma\psi^{\prime}_\lambda(u_n)\|_{\left(\mathrm{X}_0(\Omega)\right)^{\ast}}\to0\quad\text{as }n\to\infty,
\end{align}
admits a convergent subsequence in $\mathrm{X}_0(\Omega)$.

\begin{lemma}\label{lemmaPSN}
Let $\lambda>0$ and let $\mathcal{S}$ be as in \eqref{sob1}. Then, the $(PS)_c^{\mathcal N_\lambda}$ condition holds true for any $c<\frac{1}{N}\mathcal{S}^{N/p}$.
\end{lemma}
\begin{proof}
Let $\{u_n\}_n\subset\mathcal{N}_\lambda$ satisfy \eqref{palais-smale-restrict} with $c<\frac{1}{N}\mathcal{S}^{N/p}$. Since $u_n\in\mathcal N_\lambda$, so that $\langle J^{\prime}_\lambda(u_n),u_n\rangle=0$, we can argue as in Lemma \ref{limitata} to prove that $\{u_n\}_n$ is bounded in $\mathrm{X}_0(\Omega)$. From this, we can see that $\left\{\langle \psi^{\prime}_\lambda(u_n),u_n\rangle\right\}_n\subset\mathbb R$ is bounded so that, taking into account \eqref{2.11}, there exists $\ell\in(-\infty,0]$ such that, up to a subsequence, we have
\begin{equation}\label{elle}
\langle \psi^{\prime}_\lambda(u_n),u_n\rangle\to\ell\quad\text{as }n\to\infty.
\end{equation}
Furthermore, by \eqref{palais-smale-restrict} there exists $\{\gamma_n\}_n\subset\mathbb R$ such that
\begin{align}\label{palais-smale-nuova}
\|J^{\prime}_\lambda(u_n)-\gamma_n\psi^{\prime}_\lambda(u_n)\|_{\left(\mathrm{X}_0(\Omega)\right)^{\ast}}\to0\quad\text{as }n\to\infty.
\end{align}
Now, we split the proof in two situations.

\begin{itemize}
  \item If $\ell<0$, then by \eqref{palais-smale-nuova} combined with \eqref{elle} and the fact that $\{u_n\}_n\subset\mathcal N_\lambda$ is bounded, we get that $\gamma_n\to0$ as $n\to\infty$. From this fact used in \eqref{palais-smale-nuova}, we practically see that $\{u_n\}_n$ satisfies \eqref{palais-smale}, so that we conclude by Lemma \ref{lemmaPS}.

  \item If $\ell=0$, by \eqref{2.11} and \eqref{elle} we obtain
$$
\|u_n\|_q\to0\quad\mbox{and}\quad\|u_n\|_{p^{\ast}}\to0\quad\mbox{as }\,\,n\to\infty.
$$
From this, since $u_n\in\mathcal N_\lambda$ we get that $u_n\to0$ in $\mathrm{X}_0(\Omega)$ as $n\to\infty$.
\end{itemize}

\end{proof}

Now, arguing as in case $(i)$ of Lemma \ref{lemmastima} we provide the following control for $J_{\lambda}$ in $\mathcal N_\lambda$. For this, we strongly need that $m_{N,p,s}>\beta_{N,q,p}$.

\begin{lemma}\label{lemmastimanehari}
Let $\lambda>0$ and let $\mathcal{S}$ be set as in \eqref{sob1}. Then, we have
\begin{equation}\label{nehari1}
\inf_{v\in\mathcal N_\lambda} J_{\lambda}(v)<\frac{1}{N}\mathcal{S}^{N/p},
\end{equation}
and there exists $u_\lambda\in\mathcal N_\lambda$ such that
$$
J_{\lambda}(u_\lambda)=\inf_{v\in\mathcal N_\lambda} J_{\lambda}(v).
$$
\end{lemma}
\begin{proof}
Let $\varepsilon>0$ and let $v_\varepsilon$ be as in \eqref{veps}. Let us consider
$$
\psi_\lambda(tv_\varepsilon)=t^p\left(\|\nabla v_\varepsilon\|_p^p+[v_\varepsilon]_{s,p}^p\right)-\lambda t^q\|v_\varepsilon\|_q^q-t^{p^{\ast}}
=t^ph_{\varepsilon,\lambda}(t)
$$
where
$$
h_{\varepsilon,\lambda}(t):=\|\nabla v_\varepsilon\|_p^p+[v_\varepsilon]_{s,p}^p-\lambda t^{q-p}\|v_\varepsilon\|_q^q-t^{p^{\ast}-p}.
$$
Of course, $h_{\varepsilon,\lambda}$ is a continuous function with $h_{\varepsilon,\lambda}(0)>0$ and $h_{\varepsilon,\lambda}(t)\to-\infty$ as $t\to\infty$. Hence, there exists $t_{\varepsilon,\lambda}>0$ such that $h_{\varepsilon,\lambda}(t_{\varepsilon,\lambda})=0$, namely
\begin{equation}\label{4.11}
\begin{aligned}
\psi_\lambda(t_{\varepsilon,\lambda}v_\varepsilon)=\langle J^{\prime}_\lambda(t_{\varepsilon,\lambda}v_\varepsilon),t_{\varepsilon,\lambda}v_\varepsilon\rangle
=t_{\varepsilon,\lambda}^p\left(\|\nabla v_\varepsilon\|_p^p+[v_\varepsilon]_{s,p}^p\right)-\lambda t_{\varepsilon,\lambda}^q\|v_\varepsilon\|_q^q-t_{\varepsilon,\lambda}^{p^{\ast}}=0
\end{aligned}
\end{equation}
so that $t_{\varepsilon,\lambda}u_\varepsilon\in\mathcal N_\lambda$.

By \eqref{4.11}, Lemmas \ref{lemmanorma} and \ref{lemmagaglia}, we get
$$
0=t_{\varepsilon,\lambda}^p\left(\|\nabla v_\varepsilon\|_p^p+[v_\varepsilon]_{s,p}^p\right)-\lambda t_{\varepsilon,\lambda}^q\|v_\varepsilon\|_q^q-t_{\varepsilon,\lambda}^{p^{\ast}}
\leq t_{\varepsilon,\lambda}^p\left(\mathcal{S}+C\varepsilon^{m_{N,p,s}}\right)-t_{\varepsilon,\lambda}^{p^{\ast}}.
$$
from which
\begin{equation}\label{4.12}
t_{\varepsilon,\lambda}\leq\left(\mathcal{S}+C\varepsilon^{m_{N,p,s}}\right)^{1/(p^{\ast}-p)}.
\end{equation}
By \eqref{4.11}, considering \eqref{sob1} with $\|v_\varepsilon\|_{p^{\ast}}=1$, we also have
$$
\lambda t_{\varepsilon,\lambda}^{q-p}\|v_\varepsilon\|_q^q+t_{\varepsilon,\lambda}^{p^{\ast}-p}=\|\nabla v_\varepsilon\|_p^p+[v_\varepsilon]_{s,p}^p\geq \mathcal{S}
$$
so that by Sobolev embedding and Lemma \ref{lemmanorma}, we get
$$
\mathcal{S}\leq\lambda t_{\varepsilon,\lambda}^{q-p}C\left(\mathcal{S}+C\varepsilon^{(N-p)/(p-1)}\right)+t_{\varepsilon,\lambda}^{p^{\ast}-p},
$$
which implies that $t_{\varepsilon,\lambda}\geq\mu_\lambda>0$ for $\varepsilon>0$ sufficiently small. From this point, we can argue exactly as in the proof of case $(i)$ of Lemma \ref{lemmastima} to show that
$$
J_{\lambda}(t_{\varepsilon,\lambda}v_\varepsilon)<\frac{1}{N}\mathcal{S}^{N/p},
$$
which immediately yields \eqref{nehari1}.

Finally, by \eqref{palais-smale-restrict}, Lemma \ref{lemmaPSN} and \cite[Theorem 7.12]{AM}, we get that there exists $u_\lambda\in\mathcal N_\lambda$ such that
$$
J_{\lambda}(u_\lambda)=\inf_{v\in\mathcal N_\lambda} J_{\lambda}(v),
$$
concluding the proof.
\end{proof}

Let us introduce the following notation
$$
\mathcal D^{1,p}(\mathbb R^N) = \left\{u\in L^{p^{\ast}}(\mathbb R^N):\,\,\frac{\partial u}{\partial x_i}\in L^p(\mathbb R^N),\,\,\text{weak derivatives}, \,\,i=1,\ldots,N\right\}.
$$
Then, by a classical argument in \cite[Theorem I.1]{Lions}, see also \cite[Lemma 3.1]{AD}, we can state the following technical lemma.

\begin{lemma}\label{lemmalions}
Let $\{u_n\}_n\subset \mathrm{X}_0(\Omega)$ be a sequence with $\|u_n\|_{p^{\ast}}=1$ and $\|\nabla u_n\|_p^p\to \mathcal{S}$. Then, there exists a sequence $\{(y_n,\theta_n)\}_n\subset\mathbb R^N\times\mathbb R^+$ such that $\{v_n\}_n$ with
$$
v_n(x):=\theta_n^{(N-p)/p}u_n(\theta_n x+y_n)
$$
admits a convergent subsequence, still denoted by $\{v_n\}_n$, such that $v_n\to v$ in $\mathcal D^{1,p}(\mathbb R^N)$. Moreover, $\theta_n\to0$ and $y_n\to y\in\overline{\Omega}$.
\end{lemma}

Since $\Omega$ is a smooth and bounded domain of $\mathbb R^N$, we can choose $r>0$ small enough so that the following sets
$$
\Omega_r^+=\left\{x\in\mathbb R^N:\,\,\text{dist}(x,\Omega)<r\right\},\qquad\Omega_r^-=\left\{x\in\mathbb R^N:\,\,\text{dist}(x,\Omega)>r\right\}
$$
are homotopically equivalent to $\Omega$. We can also assume that $B_r(0)\subset\Omega$.

From now on, we consider
$$
\mathrm{X}_{0,rad}(B_r(0))=\left\{u\in \mathrm{X}_0(B_r(0)):\,\,u\mbox{ is radial}\right\}
$$
and we set $J_{\lambda,B_r}:\mathrm{X}_{0,rad}(B_r(0))\to\mathbb R$ given as
$$
J_{\lambda,B_r}(u)=\frac{1}{p}\left(\int_{B_r(0)}|\nabla u|^pdx+\iint_{\mathbb R^{2N}}\frac{|u(x)-u(y)|^p}{|x-y|^{N+ps}}\,dx\,dy\right)
-\frac{\lambda}{q}\int_{B_r(0)}|u|^q dx-\frac{1}{p^{\ast}}\int_{B_r(0)}|u|^{p^{\ast}}dx.
$$
Then, setting for simplicity
$$
\mathcal N_{\lambda,B_r}=\left\{v\in \mathrm{X}_{0,rad}(B_r(0))\setminus\{0\}:\,\,\langle J_{\lambda,B_r}^{\prime}(v),v\rangle=0\right\}
$$
we can easily see that all previous results proved for $J_\lambda$ are valid also for $J_{\lambda,B_r}$. In particular, considering that $v_\varepsilon$ in \eqref{veps} is radial, by Lemma \ref{lemmastimanehari} there exists $u_{\lambda,B_r}\in\mathcal N_{\lambda,B_r}$ such that
\begin{equation}\label{proprieta}
 J_{\lambda,B_r}(u_{\lambda,B_r})=\inf_{v\in\mathcal N_{\lambda,B_r}}J_{\lambda,B_r}(v)<\frac{1}{N}\mathcal{S}^{N/p}.
\end{equation}
From this function $u_{\lambda,B_r}$, denoting $m(\lambda):=J_{\lambda,B_r}(u_{\lambda,B_r})$ for simplicity, we can set the sublevel
$$
J_\lambda^{m(\lambda)}=\left\{u\in\mathcal N_\lambda:\,\,J_\lambda(u)\leq m(\lambda)\right\}
$$
and the function $\zeta_\lambda:\Omega_r^-\to J_\lambda^{m(\lambda)}$ as
\begin{equation}\label{5.10}
\zeta_\lambda(y)(x)=\begin{cases}
			u_{\lambda,B_r}(x-y) & \text{if } x\in B_r(y),\\[1ex]
			0 &\text{otherwise},
		\end{cases}
\end{equation}
for any $y\in\Omega_r^-$, which is well defined.
Let us set the barycenter map $\beta_\lambda:\mathcal N_\lambda\to\mathbb R^N$ as
$$
\beta_\lambda(u):=\frac{1}{\mathcal{S}^{N/p}}\int_{B_r(y)} x|u|^{p^{\ast}}dx.
$$
Then, for any $y\in\Omega_r^-$, we have
$$
\begin{aligned}
\left(\beta_\lambda\circ\zeta_\lambda\right)(y)&=\frac{1}{\mathcal{S}^{N/p}}\int_\Omega x |u_{\lambda,B_r}(x-y)|^{p^{\ast}}dx
=\frac{1}{\mathcal{S}^{N/p}}\int_\Omega (z+y) |u_{\lambda,B_r}(z)|^{p^{\ast}}dz\\
&=\frac{1}{\mathcal{S}^{N/p}}\int_\Omega z|u_{\lambda,B_r}(z)|^{p^{\ast}}dz+\frac{y}{\mathcal{S}^{N/p}}\int_\Omega |u_{\lambda,B_r}(z)|^{p^{\ast}}dz\\
&=y\frac{\|u_{\lambda,B_r}\|_{p^{\ast}}^{p^{\ast}}}{\mathcal{S}^{N/p}}
\end{aligned}
$$
where last equality follows from the fact that $u_{\lambda,B_r}$ is radial.

\begin{lemma}\label{lemma5.3}
There exists $\lambda_{{\ast}{\ast}}>0$ such that for any $\lambda\in(0,\lambda_{{\ast}{\ast}})$ we have that if $u\in J_\lambda^{m(\lambda)}$ then $\beta_\lambda(u)\in\Omega_r^+$.
\end{lemma}
\begin{proof}
Let us argue by contradiction, assuming there exist $\{\lambda_n\}_n\subset\mathbb R^+$ and $\{u_n\}_n\subset J_{\lambda_n}^{m(\lambda_n)}$ such that $\lambda_n\to0$ but $\beta_{\lambda_n}(u_n)\notin\Omega_r^+$. Since $u_n\in J_{\lambda_n}^{m(\lambda_n)}$, we get
\begin{equation}\label{5.36prima}
\frac{1}{p}\|u_n\|_{\mathrm{X}_0}^p-\frac{\lambda_n}{q}\|u_n\|_q^q-\frac{1}{p^{\ast}}\|u_n\|_{p^{\ast}}^{p^{\ast}}=J_{\lambda_n}(u_n)\leq m(\lambda_n)
\end{equation}
and
\begin{equation}\label{5.37prima}
\|u_n\|_{\mathrm{X}_0}^p-\lambda_n\|u_n\|_q^q-\|u_n\|_{p^{\ast}}^{p^{\ast}}=\langle J^{\prime}_{\lambda_n}(u_n),u_n\rangle=0.
\end{equation}
Combining \eqref{5.36prima} and \eqref{5.37prima}, arguing as in Lemma \ref{limitata}, we easily get the boundedness of $\{u_n\}_n\subset \mathrm{X}_0(\Omega)$ and so of $\{\|u_n\|_q\}_n\subset\mathbb R^+$. Thus, by \eqref{5.36prima} we get
\begin{equation}\label{5.36}
\frac{1}{p}\|u_n\|_{\mathrm{X}_0}^p-\frac{1}{p^{\ast}}\|u_n\|_{p^{\ast}}^{p^{\ast}}\leq m(\lambda_n)+\text{o}_n(1)
\end{equation}
and by \eqref{5.37prima}
\begin{equation}\label{5.37}
\|u_n\|_{\mathrm{X}_0}^p-\lambda_n\|u_n\|_q^q-\|u_n\|_{p^{\ast}}^{p^{\ast}}=\text{o}_n(1),
\end{equation}
as $n\to\infty$. By \eqref{5.36} and \eqref{5.37}, it follows that
$$
\frac{1}{N}\|u_n\|_{\mathrm{X}_0}^p\leq m(\lambda_n)+\text{o}_n(1)
$$
which yields, jointly with \eqref{proprieta}, that
\begin{equation}\label{5.38}
\|u_n\|_{\mathrm{X}_0}^p\leq \mathcal{S}^{N/p}+\text{o}_n(1)
\end{equation}
as $n\to\infty$.

Let us set $w_n=u_n/\|u_n\|_{p^{\ast}}$. Using \eqref{sob1}, \eqref{5.37} and \eqref{5.38}, we have
$$
\mathcal{S}\leq\|\nabla w_n\|_p^p\leq\|w_n\|_{\mathrm{X}_0}^p=\frac{\|u_n\|_{\mathrm{X}_0}^p}{\|u_n\|_{p^{\ast}}^p}\leq\|u_n\|_{\mathrm{X}_0}^{p-p^2/p^{\ast}}+\text{o}_n(1)\leq \mathcal{S}+\text{o}_n(1)
$$
as $n\to\infty$. Hence, we construct a sequence $\{w_n\}_n\subset \mathrm{X}_0(\Omega)$ such that $\|w_n\|_{p^{\ast}}=1$ and $\|\nabla w_n\|_p^p\to S$ as $n\to\infty$.
By Lemma \ref{lemmalions}, there exists a sequence $\{(y_n,\theta_n)\}_n\subset\mathbb R^N\times\mathbb R^+$ such that the sequence $\{v_n\}_n$ set as
$$
v_n(x)=\theta_n^{(N-p)/p}w_n(\theta_nx+y_n)
$$
converges, up to a subsequence, to some $v\in W^{1,p}(\mathbb R^N)$. By \eqref{5.37} we get
$$
\mathcal{S}\|u_n\|_{p^{\ast}}^p+\text{o}_n(1)=\|\nabla u_n\|_p^p\leq\|u_n\|_{\mathrm{X}_0}^p=\|u_n\|_{p^{\ast}}^{p^{\ast}}+\text{o}_n(1)
$$
which implies
\begin{equation}\label{5.41}
\|u_n\|_{p^{\ast}}^{p^{\ast}}\geq \mathcal{S}^{N/p}+\text{o}_n(1)
\end{equation}
as $n\to\infty$.
Similarly, by \eqref{5.38} we get
$$
\mathcal{S}\|u_n\|_{p^{\ast}}^p+\text{o}_n(1)=\|\nabla u_n\|_p^p\leq\|u_n\|_{\mathrm{X}_0}^p\leq \mathcal{S}^{N/p}+\text{o}_n(1)
$$
which implies
\begin{equation}\label{5.42}
\|u_n\|_{p^{\ast}}^{p^{\ast}}\leq \mathcal{S}^{N/p}+\text{o}_n(1)
\end{equation}
as $n\to\infty$.
Thus, by \eqref{5.41} and \eqref{5.42} we conclude that
\begin{equation}\label{5.43}
\|u_n\|_{p^{\ast}}^{p^{\ast}}\to \mathcal{S}^{N/p}\quad\mbox{as }n\to\infty.
\end{equation}
We observe that
$$
\beta_{\lambda_n}(u_n)=\frac{1}{\mathcal{S}^{N/p}}\int_\Omega x|u_n|^{p^{\ast}}dx=\frac{\|u_n\|_{p^{\ast}}^{p^{\ast}}}{\mathcal{S}^{N/p}}\int_\Omega x|w_n|^{p^{\ast}}dx.
$$
For $\phi\in C_0^\infty(\mathbb R^N)$ be such that $\phi(x)=x$ for all $x\in\overline{\Omega}$, we get by the Lebesgue Dominated Convergence Theorem
$$
\beta_{\lambda_n}(u_n)=\frac{\|u_n\|_{p^{\ast}}^{p^{\ast}}}{\mathcal{S}^{N/p}}\int_{\mathbb R^N} \phi(x)|w_n(x)|^{p^{\ast}}dx
=\frac{\|u_n\|_{p^{\ast}}^{p^{\ast}}}{\mathcal{S}^{N/p}}\int_{\mathbb R^N} \phi(\theta_n x+y_n)|v_n(x)|^{p^{\ast}}dx\to y\in\overline{\Omega}
$$
as $n\to\infty$, since also $y_n\to y$ by Lemma \ref{lemmalions} and considering \eqref{5.43}. This gives the desired contradiction.
\end{proof}

\begin{lemma}\label{lemma6.1}
Let $\lambda_{{\ast}{\ast}}>0$ be as in Lemma \ref{lemma5.3}. Then, for any $\lambda\in(0,\lambda_{{\ast}{\ast}})$ we have that
$$
\text{cat}_{J_\lambda^{m(\lambda)}}\left(J_\lambda^{m(\lambda)}\right)\geq cat_\Omega(\Omega).
$$
\end{lemma}
\begin{proof}
Suppose that $\text{cat}_{J_\lambda^{m(\lambda)}}\left(J_\lambda^{m(\lambda)}\right)=M$, that is
\begin{equation}\label{6.1}
J_\lambda^{m(\lambda)}=\Theta_1\cup\ldots\cup\Theta_M,
\end{equation}
where $\Theta_j$ is closed and contractible in $J_\lambda^{m(\lambda)}$, for $j=1,\ldots,M$. This means that there exist
$h_j\in C^0\left([0,1]\times\Theta_j\,;\,J_\lambda^{m(\lambda)}\right)$ and $u_j\in\Theta_j$ such that
$$
h_j(0,u)=u,\qquad h_j(1,u)=u_j,\qquad\mbox{for any }u\in\Theta_j.
$$
By \eqref{5.10} and \eqref{6.1}, we can set
$$
B_j=\zeta_\lambda^{-1}(\Theta_j),\qquad j=1,\ldots,M.
$$
In this way, $B_j$ are closed in $\Omega_r^-$ and satisfy
$$
\Omega_r^-=B_1\cup\ldots\cup B_M.
$$
Let $g_j:[0,1]\times B_j\to\Omega_r^+$ be the deformation given by
$$
g_j(t,y)=\beta_\lambda\left(h_j(t,\zeta_\lambda(y))\right),
$$
which is well defined thanks to Lemma \ref{lemma5.3}. These maps are contractions of the sets $B_j$ in $\Omega_r^+$, so
$$
\text{cat}_\Omega(\Omega)=\text{cat}_{\Omega_r^-}(\Omega_r^+)\leq M,
$$
concluding the proof.

\end{proof}

Finally, we can address the proof of  Theorem \ref{T1}.

\begin{proof}[Proof of Theorem \ref{T1}]
Let $\lambda\in(0,\lambda_{{\ast}{\ast}})$ with $\lambda_{{\ast}{\ast}}$ given in Lemma \ref{lemma5.3}.
By Lemma \ref{lemmaPSN} we know that the $(PS)_c^{\mathcal N_\lambda}$ condition holds true for any $c<\frac{1}{N}\mathcal{S}^{N/p}$.

Thus, by Lemmas \ref{lemmastimanehari} and \ref{lemma6.1} joint with the standard Lusternik-Schnirelman theory, as in \cite[Theorem 5.20]{W}, we obtain at least $\text{cat}_\Omega(\Omega)$ critical points of $J_\lambda$ restricted to $\mathcal N_\lambda$.
Of course, each of these critical points is a nontrivial critical point of the unconstrained functional $J_\lambda$.

Indeed, if $u\in\mathcal N_\lambda$ is a critical point of $J_\lambda$ restricted on $\mathcal N_\lambda$, there exists $\gamma\in\mathbb R$ such that
\begin{equation}\label{2.12}
J^{\prime}_\lambda(u)=\gamma\psi^{\prime}_\lambda(u),
\end{equation}
with $\psi_\lambda$ set as at the beginning of the present section.
In particular, we have
$$
0=\langle J^{\prime}_\lambda(u),u\rangle=\gamma\langle \psi^{\prime}_\lambda(u),u\rangle.
$$
From this and \eqref{2.11} we deduce that $\gamma=0$. Thus, by \eqref{2.12} we get that $J^{\prime}_\lambda(u)=0$.

In conclusion, we obtain at least $\text{cat}_\Omega(\Omega)$ nontrivial critical points of unconstrained $J_\lambda$, which provide $\text{cat}_\Omega(\Omega)$ nontrivial solutions of \eqref{P}.

\end{proof}

\section*{Acknowledgments}
This manuscript is part of the third author's Ph.D thesis. V.A.B. Viloria would like to thank the Department of Mathematics from Universidade Estadual de Campinas for fostering a pleasant and productive scientific atmosphere, which has benefited a lot the final outcome of this project. V.A.B. Viloria thanks to Capes-Brazil (Doctoral Scholarship). J.V. da Silva was partially supported by Conselho Nacional de Desenvolvimento Cient\'{i}fico e Tecnol\'{o}gico (CNPq-Brazil) under Grants No. 307131/2022-0 and FAPDF Demanda Espont\^{a}nea 2021 and FAPDF - Edital 09/2022 - Demanda Espont\^{a}nea.
A.\,Fiscella is member of the {Gruppo Nazionale per l'Analisi Ma\-tema\-tica, la Probabilit\`a e
	le loro Applicazioni} (GNAMPA) of the {Istituto Nazionale di Alta Matematica ``G. Severi"} (INdAM).
A.\,Fiscella realized the manuscript within the auspices of the FAPESP Thematic Project titled "Systems and partial differential equations" (2019/02512-5).

\end{document}